\documentclass[10pt]{article}       
\usepackage{geometry}               
\usepackage{graphicx}
\usepackage{caption}
\usepackage{subcaption}
\usepackage{amsmath} 
\usepackage{amssymb}  
\usepackage{amsthm}  
\usepackage{bm}
\usepackage{lipsum}
\usepackage{color}

\newtheorem{proposition}{Proposition}
\newtheorem{definition}{Definition}
\newtheorem{corollary}{Corollary}
\newtheorem{lemma}{Lemma}
\newtheorem{theorem}{Theorem}
\newtheorem{remark}{Remark}

\DeclareMathOperator*{\argmax}{arg\,max}

\title{Optimal Dynamic Contracts for a Large-Scale Principal-Agent Hierarchy: A Concavity-Preserving Approach} 

\author{Christopher W. Miller\footnotemark[1]
\and
 Insoon Yang\footnotemark[2]
}

\date{}

\begin{document}
\maketitle

\footnotetext[1]{Department of Mathematics, University of California, Berkeley
({miller@math.berkeley.edu}). Supported in part by NSF GRFP under grant number DGE 1106400.}
\footnotetext[2]{Department of Electrical Engineering and Computer Sciences,~University of California,~Berkeley ({iyang@eecs.berkeley.edu}). Supported in part by NSF CPS project FORCES under grant number 1239166.}


\begin{abstract}We present a continuous-time contract whereby a top-level player can incentivize a hierarchy of players below him to act in his best interest despite only observing the output of his direct subordinate. This paper extends Sannikov's approach  from a situation of asymmetric information between a principal and an agent to one of hierarchical information between several players. We develop an iterative algorithm for constructing an incentive compatible contract and define the correct notion of concavity which must be preserved during iteration. We identify conditions under which a dynamic programming construction of an optimal dynamic contract can be reduced to only a one-dimensional state space and one-dimensional control set, independent of the size of the hierarchy. In this sense, our results contribute to the applicability of dynamic programming on dynamic contracts for a large-scale principal-agent hierarchy.\end{abstract}



\pagestyle{myheadings}
\thispagestyle{plain}
\markboth{C. Miller and I. Yang}{Optimal Dynamic Contracts for a Principal-Agent Hierarchy}

\section{Introduction}

A principal-agent problem is a problem of optimal contracting between two parties in an uncertain environment. One player, namely the agent, may be able to influence the value of the output process with his actions. A separate party, the principal, wants to maximize the output process, which is subject to external noise.

It would be ideal for the principal to monitor the agent and enforce a strategy which is optimal to the principal. This case if often called the \emph{first best}.
However, it is often costly for the principal to monitor the agent's action. In many practical situations, the principal can only observe the output process, while the agent has perfect observations. Therefore, the agent is free to deviate from the action or control suggested by the principal, and the principal may not be able to detect this deviation. For example, the agent can shirk and attribute the resulting low output to noise.

In this setting of asymmetric information, the principal can only enforce an `\emph{incentive compatible}' control strategy upon the agent -- a strategy which maximizes the agent's expected utility given the compensation scheme in the contract. Then, the principal's optimal strategy is to design the combination of a compensation scheme and a recommended control strategy such that $(i)$ the recommended control is optimal to the agent given the compensation scheme, and $(ii)$ this combination maximizes the principal's expected utility. Such a combination constitutes a \emph{contract}, which is provided to the agent. 

This so-called \emph{second best} strategy in principal-agent problems has been one of primary interests in economic studies on contracts \cite{Laffont2002}. In particular, principal-agent problems in continuous-time have attracted considerable attention in economics, finance and engineering (e.g.,\cite{Holmstrom1987, Ou-Yang2003, Yang2014}).  Mathematically, such a principal-agent problem can be considered as a special case of Stackelberg differential games \cite{Long2010}. While many Stackelberg differential game problems remain as a challenge in general \cite{Basar1995}, there have been remarkable advances in solution approaches for principal-agent problems.

In Sannikov \cite{Sannikov2008, Sannikov2012},  a  dynamic programming approach is developed.  The key idea is to introduce another state variable, called the agent's \emph{continuation value}, which represents the agent's expected future utility conditioned upon the information up to current time. This approach is generalized  in \cite{EvansMiller2015} to handle more complicated dynamics of the output process and multiple agents. Another stream of research utilizes the stochastic maximum principle \cite{Williams2009, Cvitanic2009, CvitanicZhang2013}. This method requires us to solve an associated forward-backward stochastic differential equation, which is  a non-trivial task (e.g., \cite{Antonelli1993, Peng1999, Ma1999}).

In this paper, we consider a hierarchical principal-agent problem with asymmetric information in the form of moral hazard. This means there exists a hierarchy of players, each of which acts as a `principal' to the `agent' below them in the hierarchy.
\[P_0 \rightarrow P_1\rightarrow\cdots\rightarrow P_k\rightarrow P_{k+1}\rightarrow\cdots\rightarrow P_{N}.\]
More precisely, in this hierarchy of players, Player $k$ acts as a principal with Player ${k+1}$ as his agent. Moral hazard refers to the restriction that Player $k$ can only compensate Player ${k+1}$ based upon an observable noisy output, but not upon the actions actually performed by Player ${k+1}$. 
Due to usefulness of contractual hierarchies in modeling a organizational form of firms, governments and vertically-integrated economies,
several  contracts in a principal-agent hierarchy
have been proposed (e.g., \cite{Qian1994, Melumad1995, McAfee1995}).
However, these studies focus on incentive contracts in static settings.
A recent work \cite{Sung2015} extends Holmstrom and Milgrom's optimal linear contract \cite{Holmstrom1987} to the case of discrete-time hierarchical contracting by using first-order conditions.
To the authors' best knowledge, our work first develops a dynamic programming solution approach for designing optimal continuous-time \emph{dynamic} contracts in a principal-agent hierarchy.

The main result in continuous-time moral hazard problems between two players is that the principal can provide a payment scheme which incentivizes the agent to act in the principal's best interest. In \cite{Sannikov2008, Ekeland2014}, the success of these payment schemes depends upon certain concavity conditions on the agent's utility. The particular conditions and how to proceed when they are not met were emphasized in \cite{EvansMiller2015}.

The main idea of this paper is to construct a payment scheme for the hierarchical principal-agent model, whereby Player 0 can incentivize all players below him to act in his best interest. At a high level, the idea is to think of all players below Player 0 as an `aggregate agent' and inductively apply a generalized two player principal-agent result.
\[\begin{array}{ccc}
\text{Principal} & \rightarrow & \text{Aggregate Agent} \\
P_0 & \rightarrow & \framebox[6.5cm]{$P_1\rightarrow\cdots\rightarrow P_k\rightarrow P_{k+1}\rightarrow\cdots\rightarrow P_{N}$}
\end{array}\]

The first contribution of this work is to illustrate how Player 0 can construct an optimal dynamic contract which incentivizes the entire hierarchy to act in his best interest. Compared to \cite{Sannikov2008}, this result allows Player 0 to indirectly incentivize even those players whose output he cannot directly observe. Specifically, Player $k$ can only monitor the output of Player $k+1$ for $k=0, 1, \cdots, N-1$.
The main technical issue in this result is that, in general, the effective utilities of each player in the hierarchy will not be concave. To restore concavity, we develop a new concept called $(f, \Gamma)$-concave envelopes, which play an essential role to design an optimal dynamic contract via dynamic programming.

Secondly, this work illustrates how our hierarchical principal-agent model can resolve dimensionality issues in dynamic contract problems with multiple agents. Compared the contract constructed by a single principal controlling multiple agents directly in \cite{EvansMiller2015}, Player 0 only monitors one continuation value in a dynamic programming solution of his optimization problem. 
Furthermore, the maximization step is performed over a subset of efforts with dimension potentially less than $N$. 
We show how to construct this subset by a novel iterative algorithm.
This iterative construction allows us to characterize a stochastic optimal control problem, which is equivalent to the original contract design problem.
An important feature of this stochastic optimal control problem is that it is one-dimensional and therefore an optimal dynamic contract can be designed by solving an one-dimensional dynamic programming problem.

The final contribution of this work is to characterize conditions under which the subset of efforts which Player 0 can choose is a one-dimensional sub-manifold. In particular, we show that if each effort has a utility function which is strictly concave, then we can parameterize the admissible efforts by a function which we construct. In this case, we show Player 0 can construct an optimal dynamic contract by solving an associated Hamilton-Jacobi-Bellman (HJB) equation with only one-dimensional state space and one-dimensional admissible controls. We explicitly demonstrate this HJB characterization in the case of agents with quadratic utilities.

The rest of this paper is organized as follows. In Section \ref{sec:setting}, we provide the mathematical setting of the hierarchical principal-agent problem considered in this paper. In Section \ref{sec:gen}, we define a generalized two player principal-agent problem and provide an optimal dynamic contract for the principal. This contract is iteratively constructed in Section \ref{sec:main} to provide an optimal dynamic contract for the hierarchical principal-agent problem. In Section \ref{sec:envelope} we consider more details about the main concavity property of each agent's utility which much be preserved during iteration, called $(f, \Gamma)$-concavity. In particular, we prove a major result related to dimensionality reduction. Finally, in Section \ref{sec:ex}, we consider a specific example and illustrate the steps involved in constructing an optimal dynamic contract for the hierarchical model.

\section{Problem Formulation}\label{sec:setting}

Consider a hierarchy of $N+1$ players, each with a principal-agent relationship. We summarize the information asymmetry in the following diagram:
\[P_0 \rightarrow P_1\rightarrow\cdots\rightarrow P_k\rightarrow P_{k+1}\rightarrow\cdots\rightarrow P_{N}.\]
Note that  Player $k$ can only monitor the output of Player $k+1$ for $k=0, 1, \cdots, N-1$. Because the output is affected by external noise, Player $k$ is not able to correctly infer Player $k+1$'s action. For example, if output is low, Player $k$ cannot deduce if it is because Player $k+1$ performed little effort or due to bad luck. Each player provides payment to the player below him in the hierarchy, but the payment can only depend upon observable outputs and not the actual effort performed. 

The goal of this paper is to define a payment mechanism whereby Player 0 can incentivize all players below him to act in a prescribed manner, even though Player 0 cannot directly observe their actions. The proposed payment mechanism is in the form of a contract between the principal, Player $0$, and the aggregate agent, Players $1, \cdots, N$. The contract consists of a compensation scheme and a recommended control strategy for each agent, i.e., Player $k$ for $k=1, \cdots, N$. These two components must be designed before the contract starts. Player $0$'s interest is to design a contract such that $(i)$ given the compensation scheme each agent follows the recommended control strategy, and $(ii)$ this combination maximizes Player $0$'s expected utility.

\subsection{Mathematical Setup}

We consider a continuous-time system representing the output of each player up to a terminal time, $T>0$. Let $\left(\Omega,\mathcal{F},\mathbb{P}\right)$ be a probability space supporting a standard Brownian motion $B$ with natural filtration $\mathcal{F}^B(t)$. We also consider output processes $X_k$ with dynamics given below. We denote by $\mathcal{F}^{X_k}(t)$ the filtration generated by the process $X_k$.

The dynamics of the system are as follows:
\begin{itemize}
\item Player $N$ chooses an effort process $A_N$ which is adapted to $\mathcal{F}^B$. The output of Player $N$ is a process defined by
\[dX_{N} = A_{N} dt + \sigma dB.\]
\item For  $k = 1, \cdots, N-1$, Player $k$ chooses an effort process $A_k$ adapted to $\mathcal{F}^{X_{k+1}}$ and terminal compensation $R_{k+1}$ which is $\mathcal{F}^{X_{k+1}}(T)$-measurable. The output of Player $k$ is a process driven by
\begin{equation} \label{stoch}
\begin{split}
dX_k &= A_k dt + dX_{k+1}\\
& = \left(A_k + \cdots + A_N\right)dt + \sigma dB.
\end{split}
\end{equation}
Intuitively, $X_k$ represents the cumulative output from effort by Player $k$ and those below in the hierarchy.
\item Player 0 chooses a terminal compensation scheme $R_1$ which is $\mathcal{F}^{X_1}(T)$-measurable.
\end{itemize}

The utilities of the players are as follows:
\begin{itemize}
\item Player $N$ has some cost for effort, and receives a terminal compensation from $P_{N-1}$:
\[J^{N}\left[A_{N}, R_{N}\right] := \mathbb{E}^A\left[ \int_0^T r^{N}\left(A_{N}\right) dt + R_{N}\right].\]
\item For  $k = 1, \cdots, N-1$, Player $k$ has some cost for effort, receives a terminal compensation from Player ${k-1}$, and pays a terminal compensation to Player ${k+1}$:
\[J^k\left[A_{k:N},R_k,R_{k+1}\right] := \mathbb{E}^A\left[ \int_0^T r^k\left(A_k\right) dt + R_k + q^k\left(R_{k+1}\right)\right].\]
Note that the Player $k$'s utility can depend on  $A_{k+1:N}$  as well as $A_k$ because $A_{k}$ is adapted to $F^{X_{k+1}}$ and $X_{k+1}$ is driven by \eqref{stoch} which depends on $A_{k+1:N}$.
In other words, the utility for Player $k$ depends upon effort processes for Player $k+1$ through Player $N$ because these processes affect the distribution of $X_{k+1}$, which affects the distribution of $A_k$ directly.

\item Player 0 
has utility from the output $X_N$ and pays a terminal compensation to Player $1$:
\begin{equation} \nonumber
\begin{split}
J^0\left[A_{1:N},R_1\right] &:= 
\mathbb{E}^A\left[ \int_0^T dX_1 + q^0\left(R_1\right)\right]\\
&\; = \mathbb{E}^A\left[ \int_0^T A_1+\cdots + A_{N} dt + q^0\left(R_1\right)\right].
\end{split}
\end{equation}
\end{itemize}
The notation $\bold{A}_{k:N}$ represents the tuple $(\bold{A}_k,\cdots, \bold{A}_{N})$. When $k=1$, we will often use the notation $\bold{A}$ to denote the tuple $(\bold{A}_1,\cdots, \bold{A}_{N})$. The notation $\mathbb{E}^A$ clarifies that we are taking expectation under a measure where the dynamics of $X_k$ are evolved via the choice of effort $A$.
The set of admissible control for Player $k$ is given by
\[\mathcal{A}_k := \{ \bold{A}_k: [0,T] \to \mathbb{R} \: | \: \bold{A}_k(t) \mbox{ is $\mathcal{F}^{X_{k+1}}(t)$-progressively measurable} \}\]
for $k=1, \cdots, N-1$ and $\mathcal{A}_{N} := \{ \bold{A}_{N}: [0,T] \to \mathbb{R} \: | \: \bold{A}_N(t) \mbox{ is $\mathcal{F}^B(t)$-progressively measurable} \}$.
The set of admissible  compensation for Player $k$ is given by
\[\mathcal{R}_k := \{ \bold{R}_k \in \mathbb{R} \: | \: \bold{R}_k \mbox{ is $\mathcal{F}^{X_k}(T)$-measurable} \}\]
for $k=2, \cdots, N$ and $\mathcal{R}_1 := \{ \bold{R}_1 \in \mathbb{R} \: | \: \bold{R}_1 \mbox{ is $\mathcal{F}^{X_1}(T)$-progressively measurable} \}$. We also let $\mathcal{A} := \mathcal{A}_1\times\cdots\times\mathcal{A}_{N}$. Analogous notation is used for $\mathcal{R}$.

In this paper we assume that, for each $1\leq k\leq N$, the function $r^k$ is concave and $q^k$ is decreasing and linear. Economically, these assumptions correspond respectively to aversion to uncertainty in effort and risk-neutrality over terminal payments.

\subsection{Economic Problem}

The goal of Player 0 is to prescribe compensation schemes, $(R_1,\cdots ,R_{N})$, which allow him to choose the effort for every player, $(A_1,\cdots ,A_{N})$. In particular, he does this so as to maximize his expected utility.
However, because Player 0 cannot directly observe efforts, the choice must satisfy an \emph{incentive compatibility condition} -- that Player $k$ would not be better off $(i)$ by providing Player $k+1$ a different compensation scheme $R_{k+1}'$ or $(ii)$ by performing a different effort $A_k'$.
Furthermore, Player 0 cannot force other players to accept a contract, so the choice must satisfy an \emph{individual rationality condition}  -- that Player $k$ is paid more than a minimum required to accept the contract. We let $\bold{w}_k \in \mathbb{R}$ be the minimum required utility for Player $k$.

Player 0's optimization problem can then be formulated as follows:
\begin{equation} \label{opt}
\begin{split}
\sup\limits_{A \in \mathcal{A}, R \in \mathcal{R}}\quad  & J^0[A_{1:N},R_1] \\
\mbox{subject to} \quad  & J^1[A_{1:N},R_1,R_2]\geq J^1[A_{1:N}',R_1,R_2'] \quad \forall (A_{1:N}',R_2') \in \mathcal{A}_{1:N} \times \mathcal{R}_2 \\
& J^1[A_{1:N},R_1,R_2]\geq \bold{w}_1 \\
& \hspace{0.5in}\vdots \\
& J^k[A_{k:N},R_k,R_{k+1}]\geq J^k[A_{k:N}',R_k,R_{k+1}']\quad  \forall (A_{k:N}',R_{k+1}') \in \mathcal{A}_{k:N} \times \mathcal{R}_{k+1}\\
& J^k[A_{k:N},R_k,R_{k+1}]\geq \bold{w}_k \\
& \hspace{0.5in}\vdots \\
& J^{N}[A_{N},R_{N}] \geq  J^N[A_{N}',R_{N}]\quad \forall A_{N}' \in \mathcal{A}_{N} \\
& J^{N}[A_{N},R_{N}] \geq \bold{w}_{N},
\end{split}
\end{equation}
where the restrictions represents incentive compatibility and individual rationality conditions for the other players.

In summary, Player 0 proposes a set of recommended effort processes $A\in\mathcal{A}$ and terminal compensations $R\in\mathcal{R}$ in order to maximize his expected utility. The first two constraints refer to Player 1's incentive compatibility and individual rationality respectively -- for fixed $R_1$, the proposed choices must maximize Player 1's expected utility and meet a minimum utility for him to accept the contract. Each pair of constraints after that represent the analogous incentive compatibility and individual rationality constraints for each Player $k$.

As written, Player 0's optimization problem is non-trivial to solve primarily because it is unclear how to deal with the various incentive compatibility constraints. We will show that these constraints can be replaced by restrictions on the form of the terminal payments $R$ and admissible suggested efforts.
Furthermore, the reformulation will be proven to be exact.

\subsection{Main Result}

The main result of this paper is to extend the continuous-time formalism described for two player Stackelberg differential games by Sannikov \cite{Sannikov2008} to this hierarchical setup. We will show that for a particular choice of payment schemes, $(R_1,\cdots ,R_{N})$, the incentive compatibility conditions will automatically be satisfied.

We will show that Player 0's optimization problem \eqref{opt} is equivalent to the following optimization problem:
\[\begin{array}{rl}
\sup\limits_{A \in \mathcal{A}, Y_1 \in \mathcal{Y}}&J^0[A,R_1] \\
s.t. & A(t) \in\Gamma\subseteq\mathbb{R}^{N}, \quad t \in [0,T] \mbox{ a.e.}\\
& Y_1(t) \in(\partial\hat{\phi})\left(A_1(t)+\cdots+A_{N}(t)\right),\quad t \in [0,T] \mbox{ a.e.}
\end{array}\]
for a particular choice of $(R_1, \cdots, R_{N}) \in \mathcal{R}$, a given concave function $\hat{\phi}:\mathbb{R}\to\mathbb{R}$, and a given subset $\Gamma\subseteq\mathbb{R}^{N}$, which will be characterized. Here, $\partial\hat{\phi}$ represents the super-gradient of $\hat{\phi}$, and $\mathcal{Y}  := \{
\bold{Y}: [0,T] \to \mathbb{R} \: | \: \bold{Y}(t) \mbox{ is $\mathcal{F}^{B}(t)$-progressively measurable}
\}$.

The intuition is that Player 0 has an optimal effort process $A$ in mind. Player 0 provides Player 1 with a terminal compensation $R_1$, which incentivizes him to cause the rest of the hierarchy to follow $A = (A_1, \cdots, A_N)$. The function $\hat{\phi}$ represents the effective cost that Player 1 has to cause the hierarchy to follow $A$, and $Y_1$ represents the marginal cost.

There is not an explicit form for $\Gamma$ and $\hat{\phi}$, but the iterative algorithm defined in Theorem \ref{thm:Contract} is constructive. In Section 5, we show that if all players have utilities which are strictly concave, then $\Gamma$ is a one-dimensional manifold.

The main idea of the construction of the contract is to iteratively apply a generalization of a two player principal-agent result. The main technical condition that this generalization features over those in \cite{Ekeland2014,Sannikov2008} is $f$-concavity, which was highlighted in \cite{EvansMiller2015}. When dealing with generalized dynamics for the output process, as in the hierarchical case, $f$-concavity is the key condition an agent's utility must satisfy for a principal to incentivize him to follow a given effort. We will show how this concavity condition can be preserved as we iteratively apply the two player principal-agent result. This process leads to constraints on $(A_1,\cdots ,A_{N})$, which are encoded in the subset $\Gamma$.

This simplified problem is a standard stochastic control problem which may be solved by dynamic programming or, equivalently, by solving an associated Hamilton-Jacobi-Bellman equation. Then, Player 0 can explicitly construct a dynamic contract which is incentive compatible for all players, satisfies all individual rationality constraints, and maximizes his expected utility.

In particular, define the value function of Player 0 as
\[v(\bm{w},t) := \sup_{\substack{A\in \mathcal{A}^*, Y_1 \in \mathcal{Y}^*}}\mathbb{E}^A\left[\int_t^T \sum_{k=1}^{N} A_k(s) ds + q^0\left(R_1\right) \mid W(t)  = \bm{w} \right],\]
where
$\mathcal{A}^* := \mathcal{A} \cap \{\bold{A}: [0,T] \to \mathbb{R} \: | \: \bold{A}(t) \in \Gamma, t \in [0,T] \mbox{ a.e.}\}$, 
$\mathcal{Y}^* := \mathcal{Y} \cap \{\bold{Y}: [0,T] \to \mathbb{R} \: | \: \bold{Y}(t) \in (\partial\hat{\phi})\left( 1^\top A(t) \right), t \in [0,T] \mbox{ a.e.}  \}$, and $W(t)$ solves the following stochastic differential equation:
\[dW(t) = -\mu_1(A(t))dt - \sigma Y(t)dB.\]
The drift function $\mu_1:\mathbb{R}^{N}\to\mathbb{R}$ is defined in the iteration algorithm as a linear combination of each player's utility function. The constant $\sigma$ is the volatility of $X_1$. The interpretation of $W(t)$ is the \emph{continuation value} for Player 1 -- his remaining expected utility from following the proposed effort conditioned on the information up to time $t$.

Then, $v$ is the unique viscosity solution of the following Hamilton-Jacobi-Bellman (HJB) equation:
\[v_t + \sup_{\substack{a\in\Gamma_1 \\y\in(\partial\hat{\phi}_1)(1^\top a)}}\left[\frac{1}{2}y^2\sigma^2 v_{\bm{w}\bm{w}} -\mu_1(a) v_{\bm{w}} + 1^\top a\right] = 0,\]
with terminal condition $v(\bm{w},T) = q(\bm{w})$, where $\Gamma_1 \subseteq \Gamma$ will be characterized in the following sections.
After solving the HJB, Player 0 can then construct an optimal contract by constructing an optimal feedback control $A^*(t)\in\Gamma_1$ and $Y^*(t)\in(\partial\hat{\phi})(1^\top A^*(t))$ starting from the initial value $W(0)$ which will be specified. Furthermore, the optimal feedback control will be adapted to the correct filtrations.
Therefore, the proposed contract is implementable.

\section{Contract for Generalized Principal-Agent Problem} \label{sec:gen}

In this section, we consider a generalized principal-agent problem between two players. We propose a contract and demonstrate its incentive compatibility and optimality. We will ultimately iterate this result in the full principal-agent hierarchy in Section \ref{sec:main}.

Compared to \cite{Sannikov2008,EvansMiller2015}, this problem features generalized dynamics, lack of concavity on the agent's utility, and constraints on what efforts the agent may choose. All three of these features will arise naturally in our iterative construction of a contract for the principal-agent hierarchy.

\subsection{The Setting}

As in the hierarchical model, we consider a continuous-time system with an underlying probability space $(\Omega,\mathcal{F},\mathbb{P})$ supporting a standard Brownian motion $B$ with filtration $\mathcal{F}^B$. There are also two output processes $X_\alpha$ and $X_\pi$ with dynamics defined below. We also take special consideration of the filtrations $\mathcal{F}^{X_\alpha}$ and $\mathcal{F}^{X_\pi}$ which represent those generated by the processes $X_\alpha$ and $X_\pi$ respectively.

In this problem, the single agent chooses an $n$-dimensional effort process $A=(A_1,\cdots ,A_n)$, which is restricted to be $\mathcal{F}^B$-adapted and lie in a prescribed subset $\Gamma\subseteq\mathbb{R}^n$ almost surely. Meanwhile, the principal chooses a one-dimensional effort process $P$, which must be $\mathcal{F}^{X_\alpha}$-adapted, as well as a terminal payment to the agent, $R_\alpha$, which must be $\mathcal{F}^{X_\alpha}(T)$-measurable. The principal also receives a terminal payment, $R_\pi$, which is $\mathcal{F}^{X_\pi}(T)$-measurable -- the interpretation is that this is a utility from a player higher in the hierarchy.

The dynamics of $X_\alpha$ and $X_\pi$ are given by
\begin{equation} \nonumber
\begin{split}
dX_\alpha &= f(A) dt + \sigma dB,\\
dX_\pi &= P dt + dX_\alpha.
\end{split}
\end{equation}

The utilities for each party are given by
\begin{equation} \nonumber
\begin{split}
J^\alpha[A,R_\alpha] & :=  \mathbb{E}^A\left[ \bold{w}_0 + \int_0^T r^\alpha(A)dt + R_\alpha\right],\\
J^\pi[P,A,R_\alpha,R_\pi] & :=  \mathbb{E}^A\left[\int_0^T r^\pi(P,X_\alpha)dt + q^\pi(R_\alpha)+R_\pi\right].
\end{split}
\end{equation}

The principal's optimization problem can then be written as follows:
\begin{equation} \label{opt_gen}
\begin{split}
\sup\limits_{P \in \bar{\mathcal{P}}, A \in \bar{\mathcal{A}}, R_\alpha \in \bar{\mathcal{R}}} \quad & J^\pi[P,A,R_\alpha,R_\pi] \\
\mbox{subject to}  \quad  & J^\alpha[A,R_\alpha] \geq J^\alpha[A',R_\alpha]\quad \forall A'\in \bar{\mathcal{A}}\\
& A(t) \in\Gamma, \quad t \in [0,T] \mbox{ a.e.}\\
& J^\alpha[A,R_\alpha] \geq \bold{w}_\alpha,
\end{split}
\end{equation}
where
$\bar{\mathcal{A}} := \{
\bold{A}: [0,T] \to \mathbb{R} \: | \:
\bold{A}(t) \mbox{ is $\mathcal{F}^B(t)$-progressively measurable}
\}$,
$\bar{\mathcal{P}} := \{
\bold{P}: [0,T] \to \mathbb{R} \: | \:
\bold{P}(t) \mbox{ is $\mathcal{F}^{X_\alpha}(t)$-progressively measurable}
\}$ and
$\bar{\mathcal{R}} := \{
\bold{R} \in \mathbb{R} \: | \:
\bold{R} \mbox{ is $\mathcal{F}^{X_\alpha}(T)$-measurable}
\}$.

Intuitively, this problem represents a principal-agent problem with generalized dynamics and utilities. The agent chooses a multi-dimensional effort process to determine his output process. On the other hand, the principal chooses a one-dimensional effort process and a terminal payment based upon the agent's output process and some external noise.

The requirement that each process be adapted to different filtrations represents the information asymmetry in the problem in the form of moral hazard.
The relationship with our hierarchical problem will be that the principal represents Player $k$, while the agent represents Player $k+1$ controlling the efforts of all those players below him in the hierarchy.

\subsection{$f$-Concavity and $(f,\Gamma)$-Concave Envelopes}

In this subsection, we define the key concavity condition required for the proposed contract to be incentive compatible and optimal.

\begin{definition}Let  $f:\mathbb{R}^m\to\mathbb{R}^n$. We say a function $g:\mathbb{R}^m\to\mathbb{R}$ is \emph{$f$-concave} if there exists a concave function $\phi:\mathbb{R}^n\to\mathbb{R}$ such that
{\sf
\[g(x)=\phi(f(x)) \quad \forall x\in\mathbb{R}^m.\]
}
\end{definition}
This condition was shown in \cite{EvansMiller2015} to be key to incentive compatibility in standard principal-agent problems. In particular, it was shown that if the agent's utility was not $f$-concave, then the principal is restricted to particular values of effort depending on the so-called \emph{$f$-concave envelope} of the agent's utility.

In this paper, as we iterate along the hierarchy, there will always be restrictions on the set of admissible efforts. Therefore, we need a generalized notion of concave envelopes.
\begin{definition}Let $g:\mathbb{R}^m\to\mathbb{R}$ be some function and $\Gamma\subseteq\mathbb{R}^m$. We say $\hat{g}:\mathbb{R}^m\to\mathbb{R}$ is the \emph{$(f,\Gamma)$-concave envelope} of $g$ if
\begin{itemize}
\item $\hat{g}$ is $f$-concave,
\item $\hat{g}(x)\geq g(x)$ for all $x\in\Gamma$, and
\item If $\tilde{g}$ is an $f$-concave function such that $\tilde{g}(x)\geq g(x)$ for all $x\in\Gamma$, then $\tilde{g}(x)\geq\hat{g}(x)$ for all $x\in\Gamma$.
\end{itemize}
When convenient, we may refer to $\hat{g}$ by $\hat{\phi}\circ f$.\end{definition}

We  show in Section 5 how to construct the $(f,\Gamma)$-concave envelope of an arbitrary function. The equivalence $\hat{g}=\hat{\phi}\circ f$ refers to an auxillary concave function $\hat{\phi}$ produced by the construction.

We also explicitly define the touching set to be the subset where $g=\hat{g}$.

\begin{definition}Let $\hat{g}$ be the $(f,\Gamma )$-concave envelope of $g$. We define the corresponding \emph{touching set} to be:
\[\Gamma' := \left\{x\in\Gamma : g(x)=\hat{g}(x)\right\}.\]\end{definition}
An example of the touching set is shown in Figure \ref{fig:touching}.
Finally, we recall the definition of the super-gradient \cite{Boyd2004,Evans2010} of a concave function.
\begin{definition}Let $\phi:\mathbb{R}^n\to\mathbb{R}$ be a concave function. For any $x\in\mathbb{R}^n$, define the \emph{super-gradient} of $\phi$ at $x$ as the following set:
\[\partial\phi(x) := \left\{y\in\mathbb{R}^n : \phi(x)+y^\top \left(x'-x\right)\geq\phi(x')\;\; \forall x'\in\mathbb{R}^n\right\}.\]\end{definition}

\begin{figure}[tb] 
\begin{center}
\includegraphics[width =3.3in]{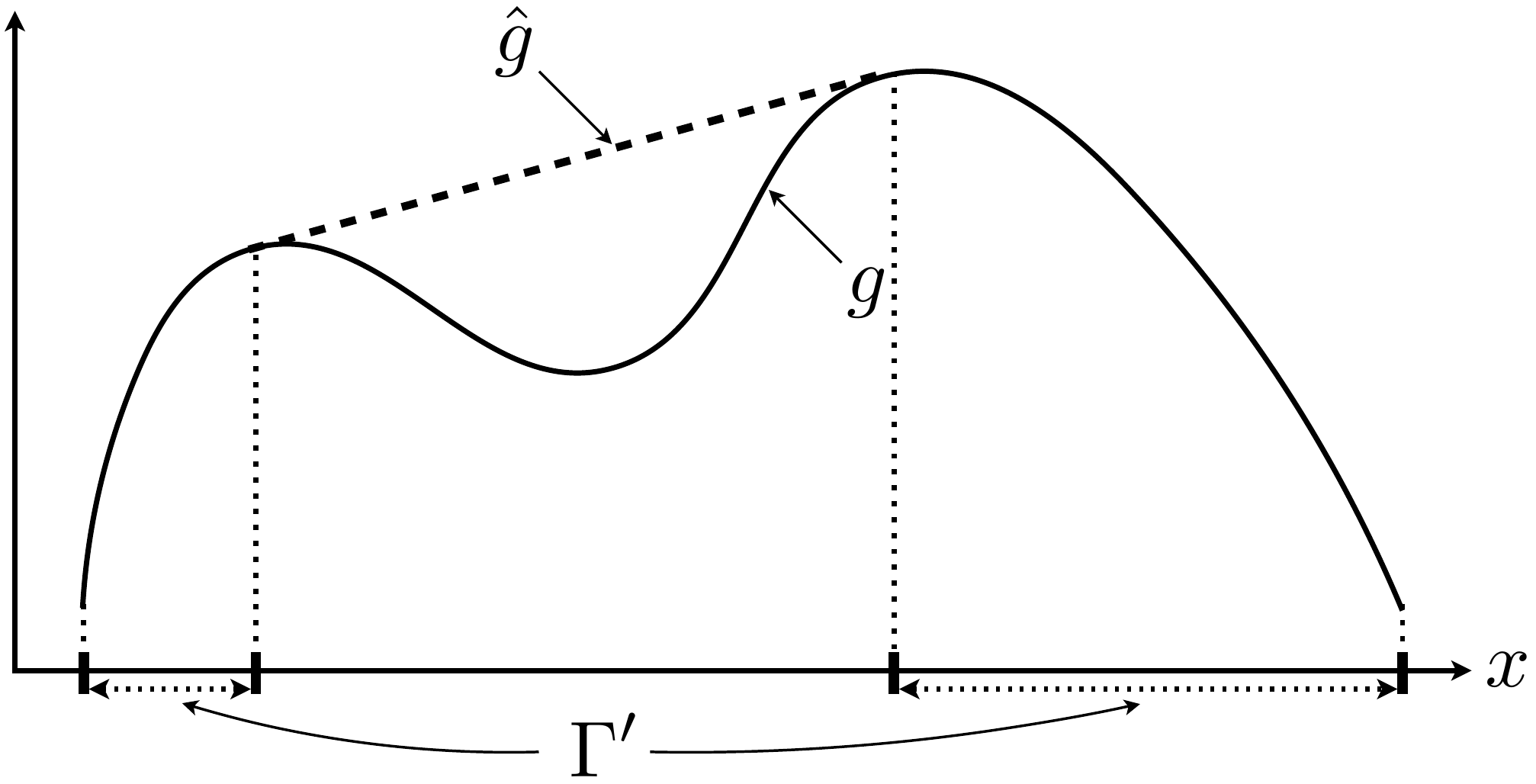}
\caption{$\hat{g}$ is the $(f, \Gamma)$-concave envelope of $g$ when $f(x) = x$ and $\Gamma = \mathbb{R}$. $\Gamma'$ is the corresponding touching set.}
 \label{fig:touching}
 \end{center}
\end{figure}

\subsection{The Contract}

We claim that for any desired choice of $(A_1,\cdots ,A_n)$, there exists a compensation $R_\alpha$ which incentivizes the agent to follow the desired effort. Furthermore, it is optimal for the principal to choose a contract of this particular form.

Let $\hat{r}^\alpha=\hat{\phi}\circ f$ be the $(f,\Gamma)$-concave envelope of $r^\alpha$, and let
\[\Gamma' := \left\{(a_1,\cdots ,a_n)\in\Gamma\mid\hat{r}^\alpha(a_1,\cdots ,a_n)=r^\alpha(a_1,\cdots ,a_n)\right\}\]
be the corresponding touching set.
We also introduce an auxiliary control variable $Y$ which plays an important role in designing a compensation $R_\alpha$ that satisfies the incentive compatibility and individual rationality. 
The set of admissible $Y$'s is given by 
$\bar{\mathcal{Y}} := \{  
\bold{Y}: [0,T] \to \mathbb{R} \: | \: \bold{Y}(t) \mbox{ is $\mathcal{F}^{B} (t)$-progressively measurable}
\}$.

\begin{proposition}Fix $P \in \bar{\mathcal{P}}$. For any choice of processes $A \in \bar{\mathcal{A}}$ such that $A(t) \in\Gamma'$ for $t \in[0,T]$ a.e., and 
$Y \in \bar{\mathcal{Y}}$ such that
$Y(t) \in (\partial\hat{\phi})\left(f(A(t))\right)$ for $t \in[0,T]$ a.e., define
\begin{equation}\label{compensation}
R_\alpha := \bold{w}_\alpha - \bold{w}_0 + \int_0^T f(A)Y - r^\alpha(A) dt - \int_0^T Y dX_\alpha.
\end{equation}
Then, $A$  is incentive compatible with $R_\alpha$, i.e.
\[J^\alpha[A,R_\alpha]\geq J^\alpha[A',R_\alpha]\quad \forall A' \in\bar{\mathcal{A}} \cap \{\bold{A} : [0,T] \to \mathbb{R} \: | \: \bold{A}(t) \in \Gamma, t \in [0,T] \mbox{ a.e.} \},
\]
and also satisfies the agent's individual rationality constraint, i.e.
\[J^\alpha[A,R_\alpha]\geq \bold{w}_\alpha.\]\end{proposition}
\begin{proof}
Suppose that the agent instead follows an arbitrary choice of effort $A'=(A_1',\cdots ,A_n') \in\bar{\mathcal{A}} \cap \{\bold{A} : [0,T] \to \mathbb{R} \: | \: \bold{A}(t) \in \Gamma, t \in [0,T] \mbox{ a.e.} \}$. Then, we compute
\begin{equation} \nonumber
\begin{split}
J^\alpha[A',R_\alpha] & :=  \mathbb{E}^{A'}\left[ \bold{w}_0 + \int_0^T r^\alpha(A')dt + R_\alpha\right]\nonumber\\
& =  \bold{w}_\alpha + \mathbb{E}^{A'}\left[\int_0^T r^\alpha(A')-\hat{r}^\alpha(A)+f(A)Y dt-\int_0^T Y dX_\alpha\right]\nonumber\\
& \leq  \bold{w}_\alpha + \mathbb{E}^{A'}\left[\int_0^T \hat{r}^\alpha(A')-\hat{r}^\alpha(A) +Y \left[f(A)-f(A')\right]dt\right]\nonumber\\
& \leq  \bold{w}_\alpha.\nonumber
\end{split}
\end{equation}
The key in this argument  is that if the agent follows $A'$ instead of $A$, then the process $X_\alpha$ evolves as
\[dX_\alpha = f(A')dt + \sigma dB^{A'},\]
where $B^{A'}$ is a $\mathbb{P}^{A'}$-Brownian motion. The existence of the new probability measure $\mathbb{P}^{A'}$ is guaranteed by the Girsanov theorem (e.g., \cite{Karatzas1991}). The rest follows from concavity using $A(t) \in\Gamma'$, $\hat{r}^\alpha\geq r^\alpha$, $\hat{r}^\alpha=\hat{\phi}\circ f$, and the definition of $\partial\hat{\phi}$. 

If the agent follows the suggested effort $A=(A_1,\cdots ,A_n)\in\bar{\mathcal{A}} \cap \{\bold{A} : [0,T] \to \mathbb{R} \: | \: \bold{A}(t) \in \Gamma', t \in [0,T] \mbox{ a.e.}\}$, we have
\begin{equation} \nonumber
\begin{split}
J^\alpha[A,R_\alpha] & :=  \mathbb{E}^A\left[\bold{w}_0 + \int_0^T r^\alpha(A)dt + R_\alpha\right]\nonumber\\
& =  \bold{w}_\alpha + \mathbb{E}^A\left[\int_0^T r^\alpha(A)-\hat{r}^\alpha(A)+f(A)Y dt-\int_0^T Y dX_\alpha\right]\nonumber\\
& =  \bold{w}_\alpha + \mathbb{E}^A\left[\int_0^T r^\alpha(A) - \hat{r}^\alpha(A)dt \right]\nonumber\\
& =  \bold{w}_\alpha \geq J^\alpha[A',R_\alpha]\quad \forall A'\in \bar{\mathcal{A}} \cap \{\bold{A} : [0,T] \to \mathbb{R} \: | \: \bold{A}(t) \in \Gamma, t \in [0,T] \mbox{ a.e.}\}.
\end{split}
\end{equation}
Here, the last equality holds due to the fact that $A (t)$ is in the touching set $\Gamma'$ for $t \in [0,T]$ a.e.
Therefore, $A$ is incentive compatible with $R_\alpha$ and the individual rationality condition is satisfied.
\end{proof}
Note that $R_\alpha$ given by \eqref{compensation} is not an admissible compensation scheme because it may not be $\mathcal{F}^{X_\alpha}(T)$-measurable. In fact, it is $\mathcal{F}^{B}(T)$-measurable.
However, we will show in the following subsection that any contract optimal to the principal has $R_\alpha$ such that it is $\mathcal{F}^{X_\alpha}(T)$-measurable and therefore is implementable.

\subsection{Optimality for Principal}

In the previous subsection, we proved that a class of contracts is incentive compatible and individually rational for the agent. In this subsection, we prove that it is sufficient for the principal to consider only contracts of this form.  In other words, the proposed class of contracts is optimal to the principal.

First, we prove two inequalities which relate $r^\alpha$ to its $(f,\Gamma )$-concave envelope $\hat{r}^\alpha=\hat{\phi}\circ f$ and the super-gradient $\partial\hat{\phi}$.

\begin{lemma}\label{lem:env}
Let $\hat{r}^\alpha=\hat{\phi}\circ f$ be the $(f,\Gamma)$-concave envelope of $r^\alpha$ and recall
\[\Gamma' := \left\{a\in\Gamma\mid r^\alpha(a)=\hat{r}^\alpha(a)\right\}.\]
Then, for any $a\in\Gamma$, we have 
\begin{enumerate}
\item If $y \not\in(\partial\hat{\phi})\left(f(a)\right)$, then there exists $a' \in \Gamma$ such that
\[\hat{r}^\alpha(a)+y (f(a' )-f(a)) < r^\alpha(a' ).\]
\item If $y \in(\partial\hat{\phi})\left(f(a)\right)$ and $a\not\in\Gamma'$, then there exists $a'\in\Gamma$ such that:
\[r^\alpha(a )+y (f(a')-f(a )) < r^\alpha(a').\]
\end{enumerate}\end{lemma}

\begin{proof}
Both proofs proceed by contradiction by building a better $(f,\Gamma)$-concave envelope than $\hat{r}^\alpha=\hat{\phi}\circ f$.
\begin{enumerate}
\item First, let $y \not \in(\partial\hat{\phi})\left(f(a )\right)$. Then, there exists $a'\in\Gamma$ such that
\[\hat{r}^\alpha(a ) + y (f(a')-f(a)) < \hat{r}^\alpha(a').\]
Furthermore, suppose that in fact for all $a'\in\Gamma$ we have
\[\hat{r}^\alpha(a)+y (f(a')-f(a)) \geq r^\alpha(a').\]
Define a function $\tilde{r}^\alpha:\mathbb{R}^n\to\mathbb{R}$ as
\[\tilde{r}^\alpha(x) := \min\left\{\hat{r}^\alpha(x),\hat{r}^\alpha(a)+y (f(x)-f(a))\right\}.\]
Then, we have
\begin{equation} \nonumber
\tilde{r}^\alpha \geq r^\alpha.
\end{equation}
Therefore, $\tilde{r}^\alpha$ is $f$-concave as the minimum of two $f$-concave functions. However, at $a'$, we have
\[\tilde{r}^\alpha(a') = \hat{r}^\alpha(a)+y (f(a')-f(a)) < \hat{r}^\alpha(a').\]
However, this contradicts the fact that $\hat{r}^\alpha$ is the $(f,\Gamma)$-concave envelope. Therefore, there exists $a'\in\Gamma$ such that
\[\hat{r}^\alpha(a)+y (f(a')-f(a)) < r^\alpha(a').\]
\item Next, let $y \in(\partial\hat{\phi})\left(f(a)\right)$, but $a \not\in\Gamma'$. Then, $r^\alpha(a) < \hat{r}^\alpha(a)$. Furthermore, suppose that in fact for all $a' \in\Gamma$ we have
\[r^\alpha(a)+y (f(a')-f(a))\geq r^\alpha(a').\]
Define a function $\tilde{r}^\alpha:\mathbb{R}^n\to\mathbb{R}$ as
\[\tilde{r}^\alpha(x) := \min\left\{\hat{r}^\alpha(x),r^\alpha(a)+y (f(x)-f(a))\right\}\]
Then, we have
\[
\tilde{r}^\alpha \geq r^\alpha.\]
Therefore, $\tilde{r}^\alpha$ is $f$-concave as the minimum of two $f$-concave functions. However, at $a$, we compute
\[\tilde{r}^\alpha(a) = r^\alpha(a) < \hat{r}^\alpha(a).\]
This is contradictory to the fact that $\hat{r}^\alpha$ is the $(f,\Gamma)$-concave envelope. Therefore, there should exist $a'\in\Gamma$ such that
$r^\alpha(a)+y (f(a')-f(a)) < r^\alpha(a')$.
\end{enumerate}
\end{proof}

Next, we demonstrate that any incentive compatible contract must be in the form of that defined above. This will suggest that it  is in fact optimal for the principal to choose a contract of this form.
\begin{proposition} \label{compensation_form}
Suppose that 
 $A \in  \bar{\mathcal{A}} \cap \{\bold{A}: [0,T] \to \mathbb{R} \: | \: \bold{A}(t) \in \Gamma, t \in [0,T] \mbox{ a.e.}\}$ is incentive compatible with $R_\alpha \in \bar{\mathcal{R}}$.  Then, the following must hold:
\begin{equation} \label{comp_prop2}
R_\alpha = w  + \int_0^T f(A)Y -r^\alpha(A)dt - \int_0^T Y dX_\alpha
\end{equation}
for some constant $w \in \mathbb{R}$,
where $Y \in \bar{\mathcal{Y}}$ such that
$Y(t) \in (\partial\hat{\phi})\left(f(A(t))\right)$ for $ t \in [0,T]$ a.e.
\end{proposition}
\begin{proof}
Let $A \in \bar{\mathcal{A}} \cap \{\bold{A}: [0,T] \to \mathbb{R} \: | \: \bold{A}(t) \in \Gamma, t \in [0,T] \mbox{ a.e.}\}$ and $R_\alpha \in \bar{\mathcal{R}}$ satisfy the incentive compatibility constraint. Define
\[V(t):=\mathbb{E}^A\left[\int_t^T r^\alpha(A)ds + R_\alpha\mid\mathcal{F}^{B}(t)\right].\]
By the Martingale representation theorem (e.g., \cite{Karatzas1991, Oksendal2003}), there exists a unique (up to a set of zero measure) process $Y \in \bar{\mathcal{Y}}$ such that
\begin{equation}\nonumber
\begin{split}
V(t) + \int_0^t r^\alpha(A)ds &= w - \int_0^t \sigma_\alpha Y dB\\
& = w + \int_0^t f(A)Yds - \int_0^t Y dX_\alpha,
\end{split}
\end{equation}
where $w = V(0)$.
Then, we have
\begin{equation}\nonumber
\begin{split}
R_\alpha &= V(T)\\
& = w + \int_0^T f(A)Y-r^\alpha(A)dt - \int_0^T Y dX_\alpha.
\end{split}
\end{equation}

Next, we show that $Y(t) \in(\partial\hat{\phi})\left(f(A(t))\right)$ for $t \in[0,T]$ a.e. Suppose that $Y(t)\not\in(\partial\hat{\phi})\left(f(A(t))\right)$ for $t \in \mathcal{T}_1$, where $\mathcal{T}_1$ is a set of a strictly positive $\mathbb{P}\times\lambda$-measure. By the first claim in Lemma \ref{lem:env}, we can find $A' \in \bar{\mathcal{A}} \cap \{\bold{A}: [0,T] \to \mathbb{R} \: | \: \bold{A}(t) \in \Gamma, t \in [0,T] \mbox{ a.e.}\}$ such that
\[\hat{r}^\alpha(A(t)) + Y(t)\left[f(A'(t))-f(A(t))\right] < r^\alpha(A'(t)) \quad \forall t \in \mathcal{T}_1.\]
Then, the agent would prefer to follow $A'$ over $A$ because
\begin{equation}
\begin{split}
J^\alpha[A',R_\alpha] & =  \mathbb{E}^{A'}\left[w_0 + \int_0^T r^\alpha(A')dt + R_\alpha\right]\nonumber\\
& =  \bold{w}_0 + w + \mathbb{E}^{A'}\left[\int_0^T r^\alpha(A')-r^\alpha(A)+ Y\left[f(A)-f(A')\right]dt\right]\nonumber\\
& \geq  \bold{w}_0 + w + \mathbb{E}^{A'}\left[\int_0^T r^\alpha(A')-\hat{r}^\alpha(A)+ Y\left[f(A)-f(A')\right]dt\right]\nonumber\\
& >  \bold{w}_0 + w + \mathbb{E}^{A'}\left[\int_0^T r^\alpha(A')-r^\alpha(A')dt\right]\nonumber\\
& =  \bold{w}_0 + w\nonumber\\
& =  J^\alpha[A,R_\alpha].\nonumber
\end{split}
\end{equation}
But this contradicts the incentive compatibility of the contract, so $Y(t)\in(\partial\hat{\phi})\left(f(A(t))\right)$ for $t \in [0,T]$ a.e.
\end{proof}

The compensation given solely by \eqref{comp_prop2} is $\mathcal{F}^{B}(T)$-measurable and may not be in $\bar{\mathcal{R}}$.
In Theorem \ref{thm:gen}, we will design an optimal contract in a relaxed space such that \eqref{comp_prop2} holds, and show that at an optimum the compensation is in $\bar{\mathcal{R}}$ and therefore is implementable.

Next, we show that the principal should only consider contracts with the suggested effort such that $A(t) \in\Gamma'$ for $t \in [0,T]$ a.e. Any other suggested effort would not be incentive compatible.

\begin{proposition}\label{prop:opt}
For any incentive compatible contract, we have
\[
 A \in \mathcal{A} \cap \{ \bold{A}: [0,T] \to \mathbb{R} \: | \: \bold{A}(t) \in \Gamma', t \in [0,T] \mbox{ a.e.} \}, \] where
\[
\Gamma' := \left\{\left(a_1,\cdots,a_n\right)\in\Gamma\mid\hat{r}^\alpha(a_1,\cdots,a_n)=r^\alpha(a_1,\cdots,a_n)\right\}.
\]
\end{proposition}

\begin{proof}
By Proposition \ref{compensation_form}, the compensation of any incentive compatible contract must be of the form
\[R_\alpha = w + \int_0^T f(A)Y-r^\alpha(A)dt - \int_0^T Y dX_\alpha\]
with $Y \in  \bar{\mathcal{Y}}$ such that $Y(t)\in(\partial\hat{\phi})\left(f(A(t))\right)$ for $t \in [0,T]$ a.e. given $A \in \bar{\mathcal{A}} \cap \{\bold{A}: [0,T] \to \mathbb{R} \: | \: \bold{A}(t) \in \Gamma, t \in [0,T] \mbox{ a.e.}\}$.
Suppose that $A (t)  \not\in \Gamma'$ for $t \in \mathcal{T}_2$, where 
$\mathcal{T}_2$ is
 a set of a strictly positive $\mathbb{P}\times\lambda$-measure. By Lemma \ref{lem:env}, we can find $A''\in \bar{\mathcal{A}} \cap \{\bold{A}: [0,T] \to \mathbb{R} \: | \: \bold{A}(t) \in \Gamma, t \in [0,T] \mbox{ a.e.}\}$ such that
\[r^\alpha(A(t)) + Y(t)\left[f(A''(t))-f(A(t))\right] < r^\alpha(A''(t)) \quad \forall t \in \mathcal{T}_2.\]
But we can check, the agent would prefer to follow $A''$ over $A$ since
\begin{eqnarray} 
J^\alpha[A'',R_\alpha] & = & \mathbb{E}^{A''}\left[\bold{w}_0 + \int_0^T r^\alpha(A'')dt + R_\alpha\right]\nonumber\\
& = & \bold{w}_0 + w + \mathbb{E}^{A''}\left[\int_0^T r^\alpha(A'')-r^\alpha(A)+ Y\left[f(A)-f(A'')\right]dt\right]\nonumber\\
& > & \bold{w}_0 + w  \nonumber \\
& = & J^\alpha[A,R_\alpha].\nonumber
\end{eqnarray}
But this contradicts the incentive compatibility of the contract, so $(A_1(t),\cdots ,A_n(t))\in\Gamma'$ for $t \in [0,T]$ a.e.
\end{proof}

Finally, because $q^\pi$ is strictly decreasing, we can argue that it suffices to consider contracts with $w=w_\alpha-w_0$. This choice makes the individual rationality constraint binding. Any larger choice of $w$ would be sub-optimal for the principal.

Up to this point, we have made no reference to the payment $R_\pi$ that the principal receives at the terminal time, which can be quite general. 
We have restricted the class of admissible compensation schemes for the agent, but which of the remaining the principal chooses will depend upon $R_\pi$. 
In particular, the principal will solve an optimization problem summarized in the following theorem:
\begin{theorem} \label{thm:gen}
Consider the following optimization problem:
\begin{equation}\label{reform_gen}
\begin{split}
\sup\limits_{P \in \bar{\mathcal{P}}, A \in \bar{\mathcal{A}}, Y \in \bar{\mathcal{Y}}} \quad & J^\pi[P,A,R_\alpha,R_\pi] \\
\mbox{subject to} \quad & A(t) \in\Gamma'\subseteq\mathbb{R}^n, \quad t \in [0,T] \mbox{ a.e.}\\
\quad  & Y(t) \in(\partial\hat{\phi})\left(f(A(t))\right), \quad  t \in [0,T] \mbox{ a.e.},
\end{split}
\end{equation}
where $R_\alpha$ is given by
\begin{equation} \label{comp_feas}
R_\alpha = \bold{w}_\alpha- \bold{w}_0 + \int_0^T f\left(A\right)Y-r^\alpha(A)dt - \int_0^T Y dX_\alpha.
\end{equation}
Then, the optimization problem \eqref{reform_gen} is equivalent to the original contract design problem \eqref{opt_gen}.
\end{theorem}
\begin{proof}
By Propositions \ref{compensation_form} and \ref{prop:opt}, incentive compatible contracts $R_\alpha$ are characterized by
\[R_\alpha = w + \int_0^T f(A)Y-r^\alpha(A)dt-\int_0^T Y dX_\alpha\]
for $w\in\mathbb{R}$, $A \in \bar{\mathcal{A}}$ such that
$A(t)\in\Gamma'\subseteq\mathbb{R}^n$ and 
$Y \in \bar{\mathcal{Y}}$ such that
$Y(t)\in(\partial\hat{\phi})\left(f(A(t))\right)$ for $t \in [0,T]$ a.e. By the same computation as in Proposition \ref{prop:opt}, we observe that $R_\alpha$ is individually rational for the agent if and only if
\[J^\alpha\left[A,R_\alpha\right] = \bold{w}_0 + w \geq \bold{w}_\alpha.\]
Now suppose $R_\alpha$ has $w> \bold{w}_\alpha- \bold{w}_0$. Then, we define a new utility,
\[R_\alpha' = \bold{w}_\alpha- \bold{w}_0 + \int_0^T f(A)Y-r^\alpha(A)dt-\int_0^T Y dX_\alpha.\]
Note that both $(A, R_\alpha)$ and $(A, R_\alpha')$ satisfy 
the incentive compatibility  and individual rationality conditions. However, we have
\begin{equation} \nonumber
\begin{split}
J^\pi\left[P,A,R_\alpha',R_\pi\right] &= J^\pi\left[P,A,R_\alpha,R_\pi\right]+\mathbb{E}^A\left[q^\pi(R_\alpha')-q^\pi(R_\alpha)\right] \\
&> J^\pi\left[P,A,R_\alpha,R_\pi\right],
\end{split}
\end{equation}
because $q^\pi$ is strictly decreasing and $R_\alpha'-R_\alpha = \bold{w}_\alpha - \bold{w}_0-w< 0$ almost surely. Therefore, we conclude it is optimal for the principal to choose $w= \bold{w}_\alpha- \bold{w}_0$.

Let $A^*$ and $Y^*$ solve the optimization problem \eqref{reform_gen}.
We now show that $R_\alpha$ given by \eqref{comp_feas} with $A^*$ and $Y^*$ is in $\bar{\mathcal{R}}$ and therefore is implementable.
The problem \eqref{reform_gen} can be rewritten as 
the following stochastic optimal control problem:
\begin{equation} \nonumber
\begin{split}
\sup_{P \in \bar{\mathcal{P}}, A \in \bar{\mathcal{A}}, Y \in \bar{\mathcal{Y}}} \quad &J^\pi [P, A, R_\alpha, R_\pi] \\
\mbox{subject to} \quad & dW = - r^\alpha (A) dt - Y dB\\
& \hspace{0.22in} = (f(A) Y - r^\alpha (A)) dt - Y dX_\alpha\\
&W(0) = \bold{w}_\alpha - \bold{w}_0\\
& R_\alpha = W(T)\\
&A(t) \in\Gamma'\subseteq\mathbb{R}^n, \quad t \in [0,T] \mbox{ a.e.}\\
\quad  & Y(t) \in(\partial\hat{\phi})\left(f(A(t))\right), \quad  t \in [0,T] \mbox{ a.e.},
\end{split}
\end{equation}
which can be solved by dynamic programming.
Its solution is given by a feedback control, $(A^*, Y^*) = (A^*(W(t)), Y^*(W(t)))$. 
Since $W$ is adapted to $\mathcal{F}^{X_\alpha}$, so are $A^*$ and $Y^*$. 
Therefore, $R_\alpha$ given by \eqref{comp_feas} with $A^*$ and $Y^*$ is $\mathcal{F}^{X_\alpha} (T)$-measurable and therefore is in $\bar{\mathcal{R}}$.
\end{proof}

In summary, we have converted a quite general two-player principal-agent with individual rationality and incentive compatibility constraints into an equivalent stochastic optimal problem with only point-wise restrictions on the values of the effort process. The idea in the remainder of the paper will be to iterate this choice with explicit choices of $R_\pi$ to solve Player 0's problem in the principal-agent hierarchy.

\section{Contract for Principal-Agent Hierarchy} \label{sec:main}

Our goal here is to inductively apply the result from above. The hope is that Player 0 can design compensation $R_1,\cdots ,R_{N}$ in such a way that Players 1 through Player $N$ will find it optimal to follow the desired efforts.
\[\begin{array}{ccc}
\text{Principal} & \rightarrow & \text{Aggregate Agent} \\
P_0 & \rightarrow & \framebox[6.5cm]{$P_1\rightarrow\cdots\rightarrow P_k\rightarrow P_{k+1}\rightarrow\cdots\rightarrow P_{N}$}
\end{array}\]
\begin{theorem}\label{thm:Contract}Suppose that $q^k(r) = -\beta_k r$ for  $k = 1, \cdots, N-1$. 
Let
\begin{equation} \label{mu1}
\begin{split}
\mu_k (A_{k:N}) :=\sum_{j=k}^N\gamma_{k,j}r^j(A_j), \quad
\mathbf{w}_k^*:=\sum_{j=k}^N\gamma_{k,j}\mathbf{w}_j
\end{split}
\end{equation}
with
$\gamma_{k,k}:=1$, $\gamma_{k,j} := \prod_{\ell=k}^{j-1}\beta_\ell$, $j =k+1, \cdots, N$, for $k=1, \cdots, N-1$.
Let 
$\hat{\phi}_N$ be the concave envelope of $r^N = \mu^N$ 
and 
$\Gamma_N := \{a_N \in \mathbb{R} \: | \: r^N (a_N) = \hat{\phi}_N (a_N)\}$.
Set
$\hat{\phi}_k$ as $(A_k + \cdots + A_N, \mathbb{R} \times \Gamma_{k+1})$-concave envelope of $\mu_k$ and $\Gamma_k \subseteq\mathbb{R}^{N-k+1}$ as the corresponding touching set for $k=N-1, \cdots, 1$.

Consider the following optimization problem:
\begin{equation} \label{reform}
\begin{array}{rl}
\sup\limits_{A \in \mathcal{A}, Y_1 \in \mathcal{Y}} &J^0[A,R_1] \\
\mbox{subject to} & A(t) \in\Gamma_1\subseteq\mathbb{R}^{N}, \quad t\in[0,T] \mbox{ a.e.}\\
& Y_1(t) \in(\partial\hat{\phi}_1)\left(A_1(t) +\cdots +A_{N}(t)\right), \quad t\in[0,T] \mbox{ a.e.},
\end{array}
\end{equation}
where
\[R_1:=\mathbf{w}_1^*+\int_0^T\left(A_1+\cdots +A_N\right)Y_1-\mu_1\left(A\right)dt - \int_0^T Y_1 dX_1.\]
%
Then, the optimization problem \eqref{reform} is equivalent to the original contract design problem \eqref{opt}.
\end{theorem}

The intuition behind this result is that by choosing a compensation scheme of the form $R_1$, Player 0 can incentivize every player below him in the hierarchy to follow a `suggested effort' process $A=(A_1,\cdots ,A_{N})$. However, it is only possible for Player 0 to incentivize the hierarchy to follow efforts $A(t) \in\Gamma_1$.

The function $\mu_1:\mathbb{R}^{N}\to\mathbb{R}$ represents the cost to Player 0 for incentivizing the hierarchy to perform $A$. The constant $\bold{w}_1^*$ represents a minimum expected utility Player 0 must provide to satisfy the entire hierarchy's individual rationality contraints. Intuitively, the function $\hat{\phi}_1$ encodes information about incentive compatibility conditions.

\begin{proof}
We proceed by induction.
In the base case, we fix a choice of $R_{N-1}$. Player $N-1$ will then solve the following optimization problem:
\[\begin{array}{rl}
\sup\limits_{\substack{A_{N-1:N} \in \mathcal{A}_{N-1:N} \\R_{N}\in \mathcal{R}_{N}}} & J^{N-1}[A_{N-1},A_{N},R_{N-1},R_{N}] \\
\mbox{subject to} & J^{N}[A_{N},R_{N}] \geq  J^N[A_{N}',R_{N}]\quad  \forall A_{N}' \in \mathcal{A}_{N}\\
& J^{N}[A_{N},R_{N}] \geq \bold{w}_{N}.
\end{array}\]
This falls in the form of the generalized two player principal-agent problem. Therefore, we can define the compensation $R_N$ as
\[R_{N} := \bold{w}_{N} + \int_0^T A_{N}Y_N -r^{N}(A_{N}) dt - \int_0^T Y_N dX_{N},\]
where $Y_N \in \mathcal{Y}$ such that $Y_N(t) \in \partial \hat{\phi}_N (A_N(t))$ for $t \in [0,T]$ a.e.,
and
$\hat{\phi}_{N}$ is the standard concave-envelope of $r^{N}$. Define the set
\[\Gamma_{N} := \{a_{N} \in \mathbb{R} \mid r^{N}(a_{N}) = \hat{\phi}_{N}(a_{N})\}.\]
Due to Theorem \ref{thm:gen}, it is equivalent for Player $N-1$ to solve
\[\begin{array}{rl}
\sup\limits_{\substack{A_{N-1:N} \in \mathcal{A}_{N-1:N} \\ Y_N\in \mathcal{Y}}}&J^{N-1}[A_{N-1},A_{N},R_{N-1},R_{N}] \\
\mbox{subject to} & A_{N}(t) \in\Gamma_{N}, \quad t\in[0,T] \mbox{ a.e.}\\
& Y_N(t) \in\partial\hat{\phi}_{N}(A_{N}(t)), \quad t\in[0,T] \mbox{ a.e.}
\end{array}\]

Next, we consider the inductive step. For fixed $R_k$, Player $k$ solves the following optimization problem:
\[\begin{array}{rl}
\sup\limits_{\substack{A_{k:N} \in \mathcal{A}_{k:N} \\
R_{k+1:N} \in \mathcal{R}_{k+1:N}}} & J^k[A_{k:N},R_k,R_{k+1}] \\
\mbox{subject to} & J^{k+1}[A_{k+1:N},R_{k+1},R_{k+2}] \geq J^{k+1}[A_{k+1:N}',R_{k+1},R_{k+2}']\\
&\hspace{1.35in} \forall A_{k+1:N}' \in \mathcal{A}_{k+1:N}, R_{k+2}' \in \mathcal{R}_{k+2} \\
& J^{k+1}[A_{k+1:N},R_{k+1},R_{k+1}] \geq \bold{w}_{k+1}\\
& \hspace{0.5in}\vdots \\
& J^{N}[A_{N},R_{N}] \geq  J^N[A_{N}',R_{N}]\quad\forall A_{N}' \in \mathcal{A}_{N}\\
& J^{N}[A_{N},R_{N}] \geq \bold{w}_{N}.\\
\end{array}\]
The inductive assumption is that it is equivalent for Player $k$ to solve
\[\begin{array}{rl}
\sup\limits_{\substack{A_{k:N} \in \mathcal{A}_{k:N} \\
Y_{k+1} \in \mathcal{Y}}}&J^k[A_{k:N},R_k,R_{k+1}] \\
\mbox{subject to} & A_{k+1:N}(t) \in\Gamma_{k+1}, \quad t\in[0,T] \mbox{ a.e.} \\
& Y_{k+1}(t) \in(\partial\hat{\phi}_{k+1})(A_{k+1}(t)+\cdots +A_{N}(t)), \quad t\in[0,T] \mbox{ a.e.},
\end{array}\]
where 
\begin{equation} \label{comp_it}
R_{k+1} = \mathbf{w}_{k+1}^* + \int_0^T (A_{k+1}+\cdots A_{N})Y_{k+1} -\mu_{k+1}(A_{k+1:N})dt - \int_0^T Y_{k+1} dX_{k+1},
\end{equation}
given $\Gamma_{k+1}\subseteq\mathbb{R}^{N-k}$, $\hat{\phi}_{k+1}:\mathbb{R}\to\mathbb{R}$ concave, and
\begin{equation}\nonumber
\begin{split}
\mu_{k+1}(A_{k+1:N}) & :=  
\sum_{\ell = k+1}^N \gamma_{k+1, j} r^{j}(A_{j})\\
\mathbf{w}_{k+1}^* & :=  \sum_{j = k+1}^N \gamma_{k+1, j} \bold{w}_{j}
\end{split}
\end{equation}
with
$\gamma_{k+1,k+1}  :=  1$ and
$\gamma_{k+1,j}  :=  \prod_{\ell=k+1}^{j-1}\beta_\ell$, $j= k+2, \cdots,  N$.

We next fix $R_{k-1}$ and focus on the analogous optimization problem for Player $k-1$. Player $k-1$ chooses $(R_k,\cdots ,R_{N})$ and $(A_{k-1},\cdots ,A_{N})$ subject to incentive compatibility and individual rationality for Player $k$ through Player $N$. By the inductive hypothesis, incentive compatibility and individual rationality for Players $k+1, \cdots, N$ is equivalent to considering $R_{k+1}$ of the form \eqref{comp_it}.
Therefore, we can write the optimization problem of Player $k-1$ as
\[\begin{array}{rl}
\sup\limits_{\substack{A_{k-1:N} \in \mathcal{A}_{k-1:N} \\ R_k \in \mathcal{R}_k}}&J^{k-1}[A_{k-1:N},R_{k-1},R_k] \\
\mbox{subject to} & J^k[A_{k:N},R_k,R_{k+1}]\geq J^k[A_{k:N}',R_k,R_{k+1}] \\& \hspace{0.2in} \forall A_{k:N}' \in \mathcal{A}_{k:N} \cup \{\bold{A}_{k:N}: [0,T] \to \mathbb{R}^{N-k+1} \: | \:
\bold{A}_{k:N}(t) \in \mathbb{R}\times\Gamma_{k+1}, t \in [0,T] \mbox{ a.e.} \}\\
& J^k[A_{k:N},R_k,R_{k+1}] \geq \bold{w}_k\\
& A_{k:N}(t) \in\mathbb{R}\times\Gamma_{k+1}, \quad t\in[0,T] \mbox{ a.e.}
\end{array}\]
We note that it is now in the form of a generalized principal-agent problem.
Furthermore, with this choice of $R_{k+1}$, we can re-write the utility for Player $k$ as
\begin{equation}\nonumber
\begin{split}
J^k[A_{k:N},R_k,R_{k+1}] & :=  \mathbb{E}^A\left[\int_0^T r^k(A_k)dt + R_k + q^k(R_{k+1})\right]\nonumber\\
&\; =  \mathbb{E}^A\left[\int_0^T r^k(A_k)dt + R_k - \beta_k R_{k+1}\right]\nonumber\\
& \; =  \mathbb{E}^A\left[-\beta_k \bold{w}_{k+1}^* + \int_0^T r^k(A_k)+\beta_k\mu_{k+1}(A_{k+1:N}) dt + R_k\right].\nonumber
\end{split}
\end{equation}
Define
\[\mu_k(A_{k:N}) := r^k(A_k)+\beta_k\mu_{k+1}(A_{k+1:N}).\]
By Theorem \ref{thm:gen}, it is equivalent for Player $k-1$ to solve:
\[\begin{array}{rl}
\sup\limits_{\substack{A_{k-1:N} \in \mathcal{A}_{k-1:N}\\ Y_k \in \mathcal{Y}}}&J^{k-1}[A_{k-1:N},R_{k-1},R_k] \\
\mbox{subject to} & A_{k:N}(t) \in\Gamma_k\subseteq\mathbb{R}^{N-k+1}, \quad t \in [0,T] \mbox{ a.e.}\\
& Y_k(t) \in(\partial\hat{\phi}_k)(A_k(t) +\cdots +A_{N}(t)), \quad t \in [0,T] \mbox{ a.e.},
\end{array}\]
where 
\[R_k := \bold{w}_k + \beta_k \bold{w}_{k+1}^* + \int_0^T (A_k+\cdots +A_{N})Y_k - \mu_k(A_{k:N})dt - \int_0^T Y_k dX_k,\]
 $\hat{\phi}_k(A_k+\cdots + A_{N})$ is the $(A_k+\cdots +A_{N},\mathbb{R}\times\Gamma_{k+1})$-concave envelope of $\mu_k$, and $\Gamma_k\subseteq\mathbb{R}^{N-k+1}$ is the corresponding touching subset.

By induction, we conclude that it is equivalent for Player 0 to solve
the following optimization problem:
\[\begin{array}{rl}
\sup\limits_{A \in \mathcal{A}, Y_1 \in \mathcal{Y}}&J^0[A,R_1] \\
\mbox{subject to} & A(t) \in\Gamma_1\subseteq\mathbb{R}^{N}, \quad t \in [0,T] \mbox{ a.e.}\\
& Y_1(t) \in(\partial\hat{\phi}_1)(A_1(t) +\cdots +A_{N}(t)), \quad t \in [0,T] \mbox{ a.e.},
\end{array}\]
where 
\[R_1 = \mathbf{w}_1^* +\int_0^T (A_1+\cdots +A_{N})Y_1-\mu_1(A)dt - \int_0^T Y_1 dX_1,\]
where $\mu_1$ and $\bold{w}_1^*$ are given by \eqref{mu1}, $\hat{\phi}_1$ is $(A_1 + \cdots + A_N, \mathbb{R} \times \Gamma_2)$-concave envelope of $\mu_1$, and $\Gamma_1\subseteq\mathbb{R}^{N}$ is the corresponding touching set.
\end{proof}

\begin{remark}Although we do not explicitly state the form of $R_2,\cdots ,R_N$ in Theorem~\ref{thm:Contract}, the details of the proof reveal that they will each be as follows:
\[R_k := \mathbf{w}_k^*+\int_0^T \left(A_k+\cdots+A_N\right)Y_k-\mu_k\left(A_{k:N}\right)dt-\int_0^T Y_k dX_k\]
where
\[\mu_k(A_{k:N}):=\sum_{j=k}^N \gamma_{k,j}r^j(A_j),\hspace{0.5cm}
\mathbf{w}_k^*:=\sum_{j=k}^N \gamma_{k,j}\mathbf{w}_j\]
with
$\gamma_{k,k} := 1$,
$k = 1, \cdots, N$
and
$\gamma_{i,j} := \prod_{\ell =i}^{j-1}\beta_\ell$, $1\leq i<j\leq N$,
and
\begin{equation} \label{y_process}
Y_k(t)\in(\partial\hat{\phi}_k)\left(A_k(t)+\cdots +A_N(t)\right),
\end{equation}
 with $\hat{\phi}_k:\mathbb{R}\to\mathbb{R}$ defined in Theorem \ref{thm:Contract}. The particular choice of $Y_2,\cdots ,Y_N$ will not affect the utility for Player 0, whenever \eqref{y_process} holds,
 because it only depends upon $R_1$.
The function $\mu_k$ represents the cost that Player $k-1$ must pay for incentivizing the hierarchy from Player $k$ to Player $N$ to follow the suggested effort $A_{k:N}$.
The scalar constant $\bold{w}_k^*$ represents a minimum expected utility that Player $k-1$ must provide to satisfy the individuality rationality constraints for Players $k, \cdots, N$.
By using the argument in Theorem \eqref{thm:Contract}, we can confirm that $R_k$ is $\mathcal{F}^{X_k}(T)$-measurable. 
Therefore, an optimal contract obtained by the proposed method is implementable. 
Furthermore, it is verifiable because an optimal $A_k$ is adapted to $\mathcal{F}^{X_k}$.
\end{remark}

We refer to the components $A_{1:N}$ and $R_{2:N}$ as `suggested' efforts and compensations because Player 0 does not have to force the other players to choose them -- we have shown it is optimal for players lower in the hierarchy to construct these contracts on their own. For this reason, we do not explicitly write them down in the theorem statement.

At this point, Player 0 can solve this equivalent optimization problem via the dynamic programming principle and  an associated Hamilton-Jacobi-Bellman equation.
Player 0's optimization problem has a running cost associated with the choice of $A$ and $Y$, as well as a terminal compensation associated with $R_1$. Then, it is natural to introduce a state variable $W$ with the property that $W(T)=R_1$:
\begin{definition}
Given a choice of processes $(A,Y)$, we define the \emph{continuation value} of the aggregate agents as
\begin{equation}
\begin{split}
W(t) & :=  \mathbf{w}_1^* + \int_0^t (A_1+\cdots +A_{N})Y-\mu_1(A)ds - \int_0^t Y dX_1\nonumber\\
& =  \mathbf{w}_1^* - \int_0^t \mu_1(A)ds - \int_0^t \sigma Y dB.\nonumber
\end{split}
\end{equation}
Then, $W(T)=R_1$.
\end{definition}

\begin{remark}The term `continuation value' refers to the characterization of $W(t)$ as the expected remaining utility for Player 1 at time $t$ conditioned on the information up to $t$, assuming he follows the suggested effort. In particular, we have
\[W(t) = R_1 + \int_t^T \mu_1(A)ds + \int_t^T \sigma Y dB.\]
Therefore,
\[W(t) = \mathbb{E}^A_t\left[\int_t^T\mu_1(A)ds + R_1\right].\]
\end{remark}
By viewing $W$ as a state variable of a (stochastic) dynamical system, we can now apply the dynamic programming principle to solve the reformulated optimization problem \eqref{reform}.
\begin{theorem}\label{thm:HJB}Define the value function for Player 0 as
\[v(\bm{w},t) := \sup_{\substack{A\in \mathcal{A}^*\\ Y\in\mathcal{Y}^*}}\mathbb{E}_t\left[\int_t^T \sum_{k=1}^{N}A_k(s) ds + q^0(W(T))\mid W(t) = \bm{w}\right].
\]
Then, $v$ is the unique viscosity solution to the following Hamilton-Jacobi-Bellman equation:
\[v_t + \sup_{\substack{a\in\Gamma_1 \\y\in(\partial\hat{\phi}_1)(1^\top a)}}\left[\frac{1}{2}y^2\sigma^2 v_{\bm{w}\bm{w}} -\mu_1 (a) v_{\bm{w}} + 1^\top a\right] = 0,\]
with terminal condition $v(\bm{w},T)=q^0(\bm{w})$.\end{theorem}
\begin{proof}
This is a standard dynamic programming argument which may be found, for example, in \cite{FlemingSoner2006,Weber2011}.
\end{proof}

If there exists a smooth enough solution to this HJB, then by standard stochastic control arguments, we can construct optimal feedback controls:
\[A(t):=A^*(W(t),t), \quad Y(t):=Y^*(W(t),t).\]
Then, using these feedback controls and the dynamics of $W(t)$, ever player in the hierarchy can compute $W$ in terms of observed values. Specifically, we have
\begin{equation} \nonumber
\begin{split}
W(t) & =  \mathbf{w}_1^* + \int_0^t \left (\sum_{i=1}^N A^*_i(W(s),s)\right)Y^*(W(s),s)-\mu\left(A^*(W(s),s)\right)ds - \int_0^t Y^*(W(s),s) dX_1\\
& =  \mathbf{w}_1^* + \int_0^t \left (\sum_{i=2}^N A^*_i(W(s),s)\right) Y^*(W(s),s)-\mu\left(A^*(W(s),s)\right)ds - \int_0^t Y^*(W(s),s) dX_2\\
& \:\: \vdots  \\
& =  \mathbf{w}_1^* + \int_0^t A^*_N(W(s),s)Y^*(W(s),s)-\mu\left(A^*(W(s),s)\right)ds - \int_0^t Y^*(W(s),s) dX_N\\
& =  \mathbf{w}_1^* - \int_0^t \mu\left(A^*(W(s),s)\right)ds - \int_0^t Y^*(W(s),s) dB.
\end{split}
\end{equation}
Namely, all three processes -- $A$, $Y$, and $W$ -- are adapted to $\mathcal{F}^{X_1},\cdots ,\mathcal{F}^{X_N}$, and $\mathcal{F}^B$.
Therefore, the proposed contract is implementable and verifiable.

Player 0 then agrees to pay Player 1 an amount $R_1 = W(T)$, where $W$ is evolved via the optimal feedback controls. Player 0 announces the contract and suggested efforts and compensations.
Then, we have shown that it is optimal for the rest of the principal-agent hierarchy to follow $A = (A_1, \cdots, A_N)$ and $R_{2:N} = (R_2, \cdots, R_N)$.

One major advantage of this hierarchical model is the dimensionality of the resulting PDE -- only one-dimensional in space. In comparison, the model with multiple agents reporting to a single principal in \cite{EvansMiller2015} resulted in a PDE with one state variable for each agent. Therefore, this hierarchical model is very tractable for even large $N$. It is worth noting, however, that this particular dimensionality result depend  on the risk-neutrality of terminal payoffs between agents in the hierarchy.
Another important feature of the proposed approach is that $\Gamma_1$ (and $\Gamma_k$'s) can be parameterized by a homeomorphism whose domain is one-dimensional when the players' utility functions are strictly concave.
This feature yields the maximization problem in the Hamiltonian of the HJB equation to be one-dimensional  at each time. 
This additional dimensionality reduction result will be provided in Theorem \ref{thm:homeo}.

\section{Details on $(f,\Gamma)$-Concave Envelopes} \label{sec:envelope}

In the previous iterative construction of an incentive compatible contract for the principal-agent hierarchy, the key condition for iteration is $(f, \Gamma)$-concavity. A simpler version of this condition was first highlighted in \cite{EvansMiller2015}, but details were not provided on how to check $(f, \Gamma)$-concavity or construct appropriate envelopes.

In this section, we provide a general construction of $(f,\Gamma )$-concave envelopes $\hat{g}$ of a function $g$, and also provide structural results about the subset where $g=\hat{g}$.

\subsection{Constructing $(f,\Gamma )$-Concave Envelopes}

Recall from the definition in Section 3 that the key properties of the $(f,\Gamma )$-concave envelope are (1) $f$-concavity, (2) majorizing the function $g$ over $\Gamma$, and (3) minimality. All three properties are intuitively represented by a step in the construction in this section.

We construct the $(f,\Gamma )$-concave envelope of $g$ in the following three steps:
\begin{enumerate}
\item Define $\phi(y):=\sup\left\{g(x)\mid x\in\Gamma\text{ s.t. }f(x)=y\right\}$, where we set $\phi(y):=-\infty$ if $f^{-1}(\left\{y\right\})\cap\Gamma=\emptyset$.
\item Compute the concave envelope $\hat{\phi}$ of $\phi$, then
\item Define $\hat{g}(x)=\hat{\phi}(f(x))$.
\end{enumerate}
An example of the proposed construction is shown in Figure  \ref{fig:construction}.

\begin{figure}[tb] 
\begin{center}
\includegraphics[width =3in]{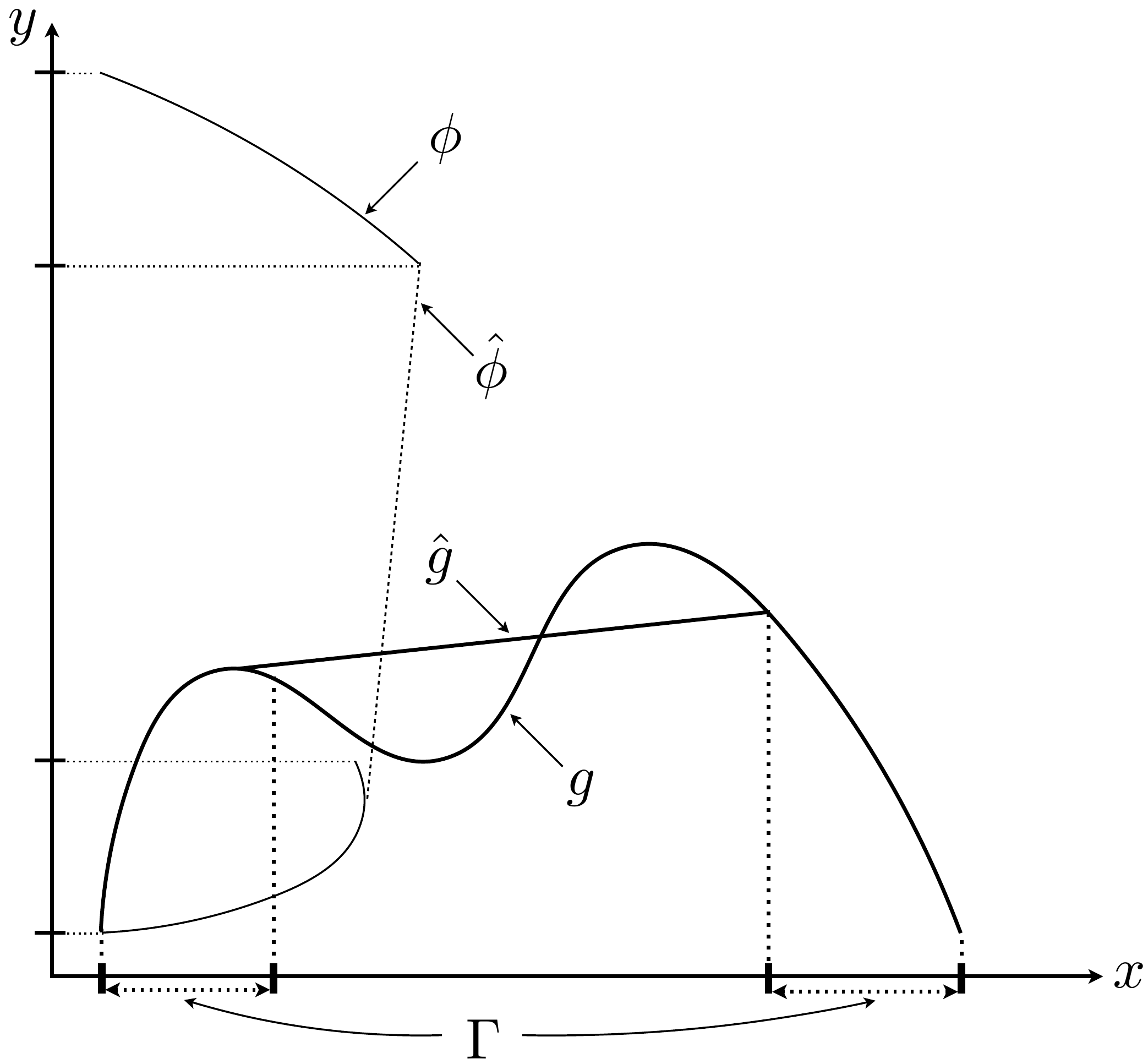}
\caption{An example of constructing $(f, \Gamma)$-concave envelope of $g$ when $f(x) = x$.}
 \label{fig:construction}
 \end{center}
\end{figure}

\begin{proposition}For any $g:\mathbb{R}^m\to\mathbb{R}$, the function $\hat{g}$, as constructed above, is the $(f,\Gamma )$-concave envelope of $g$.\end{proposition}

\begin{proof}
By construction, $\hat{g}=\hat{\phi}\circ f$ is $f$-concave and majorizes $g$ over $\Gamma$. Suppose there exists $\tilde{g}=\tilde{\phi}\circ f$ which is $f$-concave, majorizes $g$ over $\Gamma$, and is smaller than $\hat{g}$ at some point $x_0\in\Gamma$, i.e.,
\begin{equation} \label{ob1}
\tilde{\phi}(f(x_0)) < \hat{\phi}(f(x_0)).
\end{equation}
Because $\tilde{g}$ majorizes $g$ over $\Gamma$, we have
\begin{equation}\nonumber
\begin{split}
\phi(f(x)) & =  \sup\left\{g(y)\mid y\in\Gamma\text{ s.t. }f(x)=f(y)\right\}\nonumber\\
&\leq \sup\{\tilde{\phi}(f(y))\mid y\in\Gamma\text{ s.t. }f(x)=f(y)\}\nonumber\\
& = \tilde{\phi}(f(x))\hspace{0.5cm}\forall x\in\Gamma .\nonumber
\end{split}
\end{equation}
Then, recalling that $\phi(y)=-\infty$ when $f^{-1}(\left\{y\right\})\cap\Gamma =\emptyset$, we conclude that
\[\phi(y)\leq\tilde{\phi}(y)\hspace{0.5cm}\forall y\in\mathbb{R}^n,\]
which contradicts \eqref{ob1} because
$\hat{\phi}$ is the concave envelope of $\phi$.
\end{proof}

\subsection{Sufficient Conditions for $\Gamma$ to be One-Dimensional}

In this section we prove that if all players have utilities which are strictly concave in effort, then $\Gamma_1\subseteq\mathbb{R}^{N}$ for Player 0 is parameterized by a homeomorphism $\psi:\mathbb{R}\to\Gamma_1$.
The significance of this result is that it provides dimensionality reduction in Player 0's optimization problem. In particular, rather than maximizing over a potentially $N$-dimensional subset $\Gamma_1\subseteq\mathbb{R}^{N}$, Player 0 can maximize over values in $\mathbb{R}$ using the map $\psi$.
In the lack of strict concavity, the dimension of $\Gamma_1$ is not necessarily reduced.

First, we prove an abstract result.
\begin{lemma}\label{lem:cont}
Let $f:\mathbb{R}^2\to\mathbb{R}$ be continuous. Suppose that $x\mapsto f(x,y)$ is strictly concave for each $y$. Then, the function $T:\mathbb{R}\to\mathbb{R}$ defined as
\[T(y) := \argmax_{x \in \mathbb{R}} f(x,y)\]
is continuous.\end{lemma}
\begin{proof}
By strict concavity of $x\mapsto f(x,y)$ for each $y$, we conclude that $T$ is a well-defined function.

Suppose that $T$ is not continuous at $y_0$. Then there exists $\epsilon >0$ and a sequence $y_n\to y_0$ such that $|T(y_n)-T(y_0)|>\epsilon$. Define
\[\delta := f(T(y_0),y_0) - \max\left\{f(T(y_0)+\epsilon,y_0),f(T(y_0)-\epsilon,y_0)\right\} > 0,\]
which is strictly positive by strict concavity in $x$.

By continuity of $f$, we can find $y_n$ such that
\begin{equation}\nonumber
\begin{split}
|T(y_n)-T(y_0)| & >  \epsilon\\
|f(T(y_0),y_n)-f(T(y_0),y_0)| & <  \delta/4\\
|f(T(y_0)+\epsilon,y_n)-f(T(y_0)+\epsilon,y_0)| & <  \delta/4\\
|f(T(y_0)-\epsilon,y_n)-f(T(y_0)-\epsilon,y_0)| & <  \delta/4.
\end{split}
\end{equation}
Now, we verify two inequalities
\begin{equation}
\begin{split}
f(T(y_0),y_n))-f(T(y_0)+\epsilon,y_n) & >  f(T(y_0),y_0))-f(T(y_0)+\epsilon,y_0) - \delta/2\nonumber\\
& \geq  \delta - \delta/2 = \delta/2 > 0,\nonumber\\
f(T(y_0),y_n))-f(T(y_0)-\epsilon,y_n) & >  f(T(y_0),y_0))-f(T(y_0)-\epsilon,y_0) - \delta/2\nonumber\\
& \geq  \delta - \delta/2 = \delta/2 > 0.\nonumber
\end{split}
\end{equation}
But by strict concavity of $x\mapsto f(x,y_n)$, this implies
\[\argmax_{x \in \mathbb{R}} f(x,T(y_n))\in\left[T(y_0)-\epsilon,T(y_0)+\epsilon\right],\]
which contradicts $|T(y_n)-T(y_0)|>\epsilon$.
\end{proof}

Next, we prove the main argument which will be iterated.

\begin{lemma}\label{lem:par}
Let $g:\mathbb{R}^n\to\mathbb{R}$ continuous. Let $\Gamma\subseteq\mathbb{R}^{N}$ be parameterized by a homeomorphism $\psi:\mathbb{R}\to\Gamma$ such that $1^\top\psi(y)=y$ for all $y\in\mathbb{R}$. Suppose that $(y,z)\mapsto g(z,\psi(y-z))$ is jointly concave and, further, that $z\mapsto g(z,\psi(y-z))$ is strictly concave for all $y\in\mathbb{R}$.
Let $f(x_1,\cdots ,x_n) := x_1+\cdots +x_n$, $\hat{g}$ be the $(f,\mathbb{R}\times\Gamma)$-concave envelope of $g$, and the corresponding touching set be
\[\Gamma' := \left\{(x_1,\cdots ,x_n)\in\mathbb{R}\times\Gamma \: | \: g(x_1,\cdots ,x_n)=\hat{g}(x_1,\cdots ,x_n)\right\}.\]
Then, $\Gamma'$ is parameterized by a homeomorphism $\psi':\mathbb{R}\to\Gamma'$ such that $1^\top \psi'(y)=y$ for all $y\in\mathbb{R}$. Furthermore, $y\mapsto g(\psi'(y))$ is concave.\end{lemma}
\begin{proof}
We note that for every $y\in\mathbb{R}$, we can parameterize the set
\[\left\{(x_1,\cdots ,x_n)\in\mathbb{R}\times\Gamma \: | \:f(x_1,\cdots ,x_n)=y\right\}\subseteq\mathbb{R}^n\]
by the map $z\mapsto\left(z,\psi(y-z)\right)$. By the strict concavity of this map, we apply Lemma \ref{lem:cont} to conclude that the map
\[T(y) := \argmax_{z \in \mathbb{R}} g\left(z,\psi(y-z)\right)\]
is continuous.

Next, we compute
\begin{equation}\nonumber
\begin{split}
\phi(y) & :=  \sup\left\{g(x_1,\cdots ,x_n) \: | \: (x_1,\cdots ,x_n)\in\mathbb{R}\times\Gamma,\,f(x_1,\cdots,x_n)=y\right\}\nonumber\\
& \; =  \sup \left\{  g\left(z,\psi(y-z)\right) \: | \:
{z \in \mathbb{R}} \right \}.\nonumber
\end{split}
\end{equation}
By strict concavity of $z\mapsto g(z,\psi(y-z))$, we conclude that the supremum is attained at a single point for each $y$. The pointwise supremum of the jointly concave map $(y,z)\mapsto g(z,\psi(y-z))$ is concave, as shown in \cite{Boyd2004}, so $y\mapsto\phi(y)$ is concave. Therefore, $\phi = \hat{\phi}$, and $\hat{g}=\phi\circ f$.

Define a new function $\psi':\mathbb{R}\to\mathbb{R}\times\Gamma$ as:
\[\psi'(y) := \left(T(y),\psi\left(y-T(y)\right)\right).\]
By continuity of $T$, $\psi'$ is continuous. We can also check, $\phi(y) = g(\psi'(y))$, so $y\mapsto g(\psi'(y))$ is concave for all $y\in\mathbb{R}$.

We claim that the image of $\psi'$ is $\Gamma'$. By definition,
\begin{equation}
\begin{split}
\Gamma' & :=  \left\{(x_1,\cdots ,x_n)\in\mathbb{R}\times\Gamma \: | \: g(x_1,\cdots ,x_n)=\hat{g}(x_1,\cdots ,x_n)\right\}\nonumber\\
& \; =  \left\{(x_1,\cdots ,x_n)\in\mathbb{R}\times\Gamma \: | \: g(x_1,\cdots ,x_n)=\phi\circ f(x_1,\cdots ,x_n)\right\}\nonumber\\
& \;=  \bigcup_{y \in \mathbb{R}}\left\{(x_1,\cdots,x_n)\in\mathbb{R}\times\Gamma,\,f(x_1,\cdots ,x_n) = y \mid g(x_1,\cdots ,x_n)=\phi(y)\right\}\nonumber\\
& \;=  \bigcup_{y \in \mathbb{R}} \left\{(T(y),\psi(y-T(y)))\right\}\nonumber\\
& \;=  \text{Im}(\psi').\nonumber
\end{split}
\end{equation}
Finally, we confirm that
\[1^\top\psi'(y) = T(y)+1^\top\psi(y-T(y))=y,\]
which implies $\psi'$ is one-to-one and an open map, so we conclude $\psi':\mathbb{R}\to\Gamma'$ is a homeomorphism.
\end{proof}

The intuition of the following main result is to apply Lemma \ref{lem:par} iteratively. At each step, $g$ represents the effective utility of Player $k$ over $A_{k:N}$. By induction, we assume $\Gamma_{k+1}$ is one-dimensional and parametrized by a homeomorphism $\psi_{k+1}$, so Player $k$ effectively maximizes over $\mathbb{R}^2$. We show under concavity assumptions that $\Gamma_k$ is then once again one-dimensional.

\begin{theorem} \label{thm:homeo}
If each player's utility function is strictly concave and $\beta_k\geq 0$, then $\Gamma_1\subseteq\mathbb{R}^{N}$ is parameterized by a homeomorphism $\psi_1:\mathbb{R}\to\Gamma_1$.\end{theorem}
\begin{proof}
First, recall from the iteration argument that at each stage, $\mu_k:\mathbb{R}^{N-k}\to\mathbb{R}$ is a linear combination of $r^k,\cdots ,r^{N}$ with positive coefficients. Furthermore, at each step, $\Gamma_k$ is the touching set corresponding to the $(A_k+\cdots +A_{N},\mathbb{R}\times\Gamma_{k+1})$-concave envelope of $\mu_k$.

We want to show that for each $k$, $\Gamma_k$ is parameterized by a homeomorphism $\psi_k:\mathbb{R}\to\Gamma_k$ such that $1^\top\psi_k(y)=y$ and $y\mapsto\mu_k(\psi_k(y))$ is concave. We proceed by induction on $k$.

In the base case, we know $\Gamma_{N}=\mathbb{R}$, and this can be parameterized by the identity map $\psi_{N}(y)=y$. The concavity condition holds by concavity of $r^{N}$.

Now, consider the Player $k$'s effective utility function, which is given by
\[\mu_k(A_{k:N}) = r^k(A_k) + \beta_{k+1}\mu_{k+1}(A_{k+1:N}).\]
By the inductive hypothesis, there is a bijection $\psi_{k+1}:\mathbb{R}\to\Gamma_{k+1}$ such that $1^\top\psi_{k+1}(y)=y$ and $y\mapsto\mu_{k+1}(\psi_{k+1}(y))$ is concave.

In particular, we can confirm that
\[\mu_k(z,\psi_{k+1}(y-z)) = r^k(z) + \beta^{k+1}\mu_{k+1}(\psi_{k+1}(y-z))\]
is strictly concave in $z$ and jointly concave in $(y,z)$. The right-hand-side is jointly concave in $(y,z)$ as the composition of an affine function with $\mu_{k+1}\circ\psi_{k+1}$, which is concave by the inductive hypothesis. Because $r^k$ is strictly concave, we determine the sum is strictly concave in $z$ and jointly concave in $(y,z)$.

Furthermore, by the inductive hypothesis $1^\top\psi_{k+1}(y)=y$.
Therefore, we can apply Lemma \ref{lem:par} to conclude the $(A_k+\cdots +A_{N},\mathbb{R}\times\Gamma_{k+1})$-concave envelope of $\mu_k$ has a corresponding touching set $\Gamma_k$ which is parameterized by a homeomorphism $\psi_k:\mathbb{R}\to\Gamma_k$ such that $1^\top\psi_k(y)=y$ and $y\mapsto \mu_k(\psi_k(y))$ is concave. Then, the result follows by induction.
\end{proof}

The significance of this result is in reducing dimensionality in Player 0's  optimization problem at each time step of dynamic programming
 from potentially $N$-dimensional to one-dimensional. In particular, Player 0 can maximize any function $\phi:\Gamma_1\to\mathbb{R}$ by maximizing $\phi\circ\psi_1:\mathbb{R}\to\mathbb{R}$.

Practically, this result means that in the case of strictly concave utilities, not only is the HJB PDE only one-dimensional in space, but the maximization step for computing an optimal Hamiltonian at each time is also one-dimensional. This means that once $\psi_1$ is computed, Player 0's optimal dynamic contract can be constructed in time that does not depend upon the size of the hierarchy. This result makes the model tractable even for very large $N$ and contributes significantly to the tractability of using dynamic programming to design an optimal  contract in a large-scale principal-agent hierarchy.

\section{Example: Quadratic Utilities} \label{sec:ex}

In this section, we work through a specific example, where all agents have quadratic utilities. In particular, we set
\begin{equation} \nonumber
\begin{split}
r^k(a) &:= -\frac{1}{2}\alpha_k^2 a^2\\
q^k(r) &:= -\beta_k r\\
w_k & :=  0
\end{split}
\end{equation}
with $\beta_k\geq 0$ and $\alpha_k > 0$ for all $1\leq k\leq N$. Finally, we choose
\[q^0(r) := -\frac{1}{2}\alpha_0^2 r^2.\]
Economically, this means each player $1\leq k\leq N$ has an increasing aversion to putting in effort personally, which is governed by the parameter $\alpha_k$. Furthermore, each player $1\leq k\leq N$ has some aversion to paying the player below him in the hierarchy, which is governed by the parameter $\beta_k$. Player 0 is averse to uncertainty in the terminal payment to players below him.

In order to construct an optimal dynamic contract, we need to explicitly perform the iterative construction of Section 4. Because each utility function is assumed strictly concave, we can also apply the results of Section 5 to obtain a homeomorphism $\psi_k:\mathbb{R}\to\Gamma_k$ at each step. In the quadratic utility case, it will turn out that each $\Gamma_k$ will be a one-dimensional linear subspace. We will see that the iteration process only requires linear algebraic computations at each step.

\begin{theorem}There exists a strictly positive-definite, diagonal matrix $\Lambda_1$, a vector $g_1$ with $1^\top g_1 = 1$, and a real number $w_1^*$ such that it is equivalent for Player 0 to solve the following optimization problem:
\[\begin{array}{rl}
\sup\limits_{A \in \mathcal{A}, Y_1  \in \mathcal{Y}}&J^0[A,R_1] \\
\mbox{subject to} & A(t) \in\Gamma_1\subseteq\mathbb{R}^{N}, \quad t \in [0,T] \mbox{ a.e.}\\
& Y_1(t) \in(\partial\hat{\phi}_1)\left(A_1(t)+\cdots +A_{N}(t)\right) \quad t \in [0,T] \mbox{ a.e.},
\end{array}\]
with the compensation scheme given by
\[R_1 := w_1^* +\int_0^T\left(A_1+\cdots +A_{N}\right)Y_1-\mu_1\left(A\right)dt - \int_0^T Y_1 dX_1\]
where:
\begin{itemize}
\item $\mu_1(x) = -\frac{1}{2}x^\top\Lambda_1 x$,
\item $\hat{\phi_1}(y) = -\frac{1}{2}g_1^\top\Lambda_1 g_1 y^2$, so $\partial\hat{\phi}(y)=-g_1^\top\Lambda_1 g_1y$, and
\item $\Gamma_1$ is a one-dimensional subspace parametrized the isomorphism $\psi_1(y)=g_1 y$.\\
\end{itemize}
\end{theorem}

\begin{proof}
We follow the iteration argument of Theorems \ref{thm:Contract} and \ref{thm:homeo} to construct each $\mu_k$, $\Gamma_k$, and the homeomorphism $\psi_k:\mathbb{R}\to\Gamma_k$.
For Player $N$, we trivially have that $\mu_{N}(x)=-\frac{1}{2}\alpha_{N}^2x^2$, $\Gamma_{N}=\mathbb{R}$, and the homeomorphism is $\psi_{N}:\mathbb{R}\to\Gamma_{N}$ is the identity $\psi_{N}(y) = y$.

Consider Player $k$ with $k>N$. We know by Theorems \ref{thm:Contract} and \ref{thm:homeo} that there exist $\mu_{k+1}:\mathbb{R}^{N-k-1}\to\mathbb{R}$, $\Gamma_{k+1}\subseteq\mathbb{R}^{N-k-1}$, and a homeomorphism $\psi_{k+1}:\mathbb{R}\to\Gamma_{k+1}$ such that $1^\top\psi_{k+1}(y)=y$.

We suppose further that there exists a strictly positive-definite matrix $\Lambda_{k+1}$ and vector $g_{k+1}$ with the property $1^\top g_{k+1}=1$ such that $\mu_{k+1}(x) = -\frac{1}{2}x^\top\Lambda_{k+1}x$ and $\psi_{k+1}(y)=g_{k+1}y$.

Then by Theorem \ref{thm:Contract}, $\mu_k:\mathbb{R}^{N-k-1}\to\mathbb{R}$ can be written:
\[\mu_k(x,y) = r^k(x)+\beta_{k+1}\mu_{k+1}(y)=-\frac{1}{2}\left[\begin{array}{c}x\\y\end{array}\right]^\top\left[\begin{array}{cc}\alpha_k^2 & 0\\0 & \beta_{k+1}\Lambda_{k+1}\end{array}\right]\left[\begin{array}{c}x\\y\end{array}\right],\]
or in the form $\mu_k(x)=-\frac{1}{2}x^\top\Lambda_k x$ for a strictly positive-definite, diagonal matrix.

By Theorem \ref{thm:homeo}, we need to compute the functions $T_k:\mathbb{R}\to\mathbb{R}$, $\psi_k:\mathbb{R}\to\mathbb{R}^{N-k}$ and $\phi_k:\mathbb{R}\to\mathbb{R}$. We can compute:
\begin{equation}\nonumber
\begin{split}
T_k(y) & :=  \argmax_{z \in \mathbb{R}} \mu_k(z,\psi_{k+1}(y-z)) \nonumber\\
& \; =  \argmax_z\left[-\frac{1}{2}\alpha_k^2 z^2 - \frac{1}{2}\beta_{k+1}(y-z)^2g_{k+1}^\top\Lambda_{k+1}g_{k+1}\right]\nonumber\\
& \; =  \frac{\beta_{k+1} g_{k+1}^\top\Lambda_{k+1}g_{k+1}}{\alpha_k^2+\beta_{k+1} g_{k+1}^\top\Lambda_{k+1}g_{k+1}}y,\nonumber
\end{split}
\end{equation}
and then immediately, we have
\[\psi_k(y) := \left(T_k(y),\psi_{k+1}(y-T_k(y))\right) = g_k y\]
for a vector $g_k$ computable in terms of $T_k$ and $\psi_{k+1}$. Further, we have $y=1^\top\psi_k(y)=1^\top g_k y$, so $1^\top g_k=1$.

Lastly, we have
\[\phi_k(y) := \sup_{z \in \mathbb{R}} \mu_k(z,\psi_{k+1}(y-z)) = \mu_k(\psi_k(y)) = -\frac{1}{2}y^2 g_k^\top\Lambda_k g_k.\]

By induction, the result follows. We note that because $\psi_1:\mathbb{R}\to\Gamma_1$ is linear, then $\Gamma_1$ is a one-dimensional linear subspace.
\end{proof}

Then, as in Theorem~\ref{thm:HJB}, Player 0 can define a value function and characterize it as the unique viscosity solution to an HJB equation. However, in this case, the HJB equation simplifies considerably:
\begin{corollary}Define the value function $v$ for Player 0 as in Theorem \ref{thm:Contract}. Then $v$ is the unique viscosity solution to the Hamilton-Jacobi-Bellman equation:
\[v_t + \sup_{y\in\mathbb{R}}\left[\frac{\gamma}{2}(\gamma\sigma^2v_{\bm{w}\bm{w}}+v_{\bm{w}}) y^2 + y\right] = 0,\]
with terminal condition $v(\bm{w},T)=q^0(\bm{w})$ and $\gamma := g_1^\top\Lambda_1g_1$.\end{corollary}
The proof is just a particular case of Theorem~\ref{thm:HJB}. The significance is that, in this case, the HJB reduces to only one state variable and a one-dimensional maximization over $\mathbb{R}$ at each time.


\section*{Acknowledgement}

The authors would like to thank Professor Lawrence Craig Evans for helpful discussions on PDE approaches for principal-agent problems.
 
\bibliographystyle{siam}

\bibliography{hierarchy_contract}
\vfill\eject

\end{document}